\documentclass{amsart}
\usepackage[foot]{amsaddr}
\usepackage{amssymb,latexsym,amsmath,amsfonts,graphicx,xcolor,pb-diagram}
\usepackage[mathscr]{eucal}
\usepackage{tikz, tikz-cd}
\usepackage{relsize}
\usepackage{slashed}
\newcommand{\mpar}{\par \medskip \par }
\newcommand{\Div}{\operatorname{div}}
\newcommand{\Image}{\operatorname{im}}
\newcommand{\Ker}{\operatorname{ker}}

\newcommand{\DIV}{\operatorname{Div}}
\newcommand{\Hess}{\operatorname{Hess}}
\newcommand{\Dsp}{\operatorname{dsp}}

\newcommand{\Deg}{\operatorname{deg}}

\numberwithin{equation}{section}

\swapnumbers

\theoremstyle{definition}

\newtheorem{theorem}{Theorem}[section]
\newtheorem{lemma}[theorem]{Lemma}
\newtheorem{corollary}[theorem]{Corollary}
\newtheorem{proposition}[theorem]{Proposition}

\newtheorem{definition}[theorem]{Definition}

\newtheorem{remark}[theorem]{Remark}
\newtheorem{example}[theorem]{Example}

\makeindex

\begin{document}
\title{Second Order Differential Operators on Graphs}

\author{Peter March}

\begin{abstract}
The commutator $[X, Y]$ of a pair of vector fields on a graph $G$ is not a vector field in general, but rather a second order differential operator. We investigate this departure from the classical case of vector fields on a manifold by examining the geometry of balls of radius two in $G,$ concentrating on the set of paths of length two connecting a given vertex with the center of the ball. There is a natural surjection $a\colon\mathcal{X}(T^2(G))\to DOp^2(G)$ from the space of sections of the second tangent bundle to the space of second order differential operators, whose kernel reflects the geometry of these balls. Using this map we derive bounds on the dimension of $DOp^2(G)$ in terms of the cyclomatic numbers of balls of radius two, find canonical forms for $DOp^2(G),$ provide formulas for the adjoint operator $a(Z)^*\negthinspace,$ where $Z\in\mathcal{X}(T^2(G)),$ and give necessary and sufficient conditions on $Z$ that $a(Z)^*\in DOp^2(G)$ and that $a(Z)$ is self adjoint. Consequences of this work are a Helmholtz-type decomposition of $\mathcal{X}(T^2(G))$ and a necessary and sufficient condition for $[X, Y]$ to be a vector field.
\end{abstract}

\address{Department of Mathematics\\
Rutgers University\\
Hill Center - Busch Campus\\
110 Frelinghuysen Road\\
Piscataway, NJ 08854-8019}

\email{march@math.rutgers.edu}

\maketitle
\section{Introduction}
Our goal is to use the intrinsic geometry of a graph to understand the structure of second order differential operators acting on functions of the vertices of a graph $G.$ 
\mpar
There is a considerable literature devoted to studying operators on graphs that has several distinct threads, such as scientific computing \cite{A}, \cite{AFW}, image processing \cite{C}, \cite{G}, and medicine \cite{LECCS}. Generally speaking, the kinds of graphs considered in the literature, and the kinds of operators on them, are tailored to the needs of the various disciplines. 

\mpar
In contrast to that literature, our motivation is entirely geometric. We are interested in studying linear operators $L$ defined on a finite simple graph $G$ where the value of $L\phi(i)$ only depends on the values of $\phi(j)-\phi(i),$ for $d(i,j)=1,2.$ Since we make no structural assumptions about $G$ we have to develop intrinsic ways of expressing differences of functions across edges and calculating linear combinations of differences of such differences.

\mpar
To that end, we use ideas, terminology, and results of \cite{M1} and \cite{M2} that were developed within the framework of tangent graphs. We recall these ideas sequentially in the narrative as needed. However, since the tangent graph is not a standard notion, we also provide a summary in Appendix A of relevant graph theoretic material as expressed within this framework. For the reader's convenience, we italicize the first appearance of new terms as a way of referring them to the appendix if the terms are not familiar. 

\mpar
Recall that \textit{first order differential operators} $L$ on a finite simple graph $G=(V_G, E_G)$ are linear transformations $L\colon C(G)\to\mathbb{R}$ of the form,
$$
L\phi(i) = \sum_{j\in V_G} L(i,j)\phi(j),
$$
such that, (1) $L(i,j)=0$ if $d(i,j) > 1,$ where $d(i,j)$ is the length of the shortest walk in $G$ between $i$ and $j,$ and (2) $L$ annihilates functions constant on each connected component of $G.$ Observe that item (2) is equivalent to the condition that $\sum_{j\in V_G} L(i,j)=0$ for every vertex $i$ in $G$.

\mpar
First order differential operators $L\in DOp^1(G)$ are in one-to-one correspondence with \textit{vector fields} $X\in\mathcal{X}(G)$ on $G$ via the formula,
$$
L\phi(i)=\sum_{j\in V_G}L(i,j)\phi(j)=\sum_{\pi(u)=i}X(u)d\phi(u) =X\phi(i).
$$
Let us pause to elaborate the terms in the formula above. Here $u$ is a vertex of the \textit{tangent graph} $tG$ and each such vertex is an ordered pair of vertices,
$$
u=(i,j)=ij
$$ 
for some edge $\{i,j\}$ of $G.$ Observe there are two vertices of $tG$ for each edge of $G$ and recall that $\{ij,kl\}$ is an edge in $tG$ provided $j=k$ or $i=l,$ meaning two directed edges of $G$ are adjacent in $tG$ provided the end point of one is the base point of the other. There are natural projections $\pi, \pi_+\colon V_{tG}\to V_G$ defined by the rules $\pi(ij)=i$ and $\pi_+(ij)=j,$ and an involution $\sigma\colon V_{tG}\to V_{tG}$ given by the formula $\sigma(ij)=ji.$ All three of these maps are graph homomorphisms. In these terms,
$$
d\phi(u)=\phi(\pi_+(u))-\phi(\pi(u))
$$
is the difference of $\phi$ across a directed edge and the second sum above is over all directed edges with base point $i\in V_G.$
\mpar

So, a vector field is a function $X\colon V_{tG}\to\mathbb{R},$ and vice versa, and therefore,
$$
DOp^1(G)\cong\mathcal{X}(G)\cong C(tG).
$$

\mpar
Recall that \textit{second order differential operators} are linear transformations, 
$$
L\phi(i)=\sum_{j\in V_G} L(i,j)\phi(j)
$$ 
such that (1) $L(i,j)=0$ if $d(i,j)>2$  and (2) $L$ annihilates functions constant on each connected component of $G.$ 

\mpar
It's natural to ask if the space of all such operators, $DOp^2(G),$ can be characterized in geometric terms in the same sense as $DOp^1(G).$ The answer is a qualified yes, and understanding the qualifications is the purpose of this work. The main idea is that vector fields on the tangent graph $tG,$ or equivalently sections of the \textit{second tangent bundle} $\mathcal{X}(T^2(G)),$ project to second order differential operators on $G.$

\mpar
In Section 2 we consider two concrete examples of such operators that we hope will excite the reader's interest. First, it's a simple observation that if $X$ and $Y$ are vector fields then they can be composed as first order differential operators. We show that the commutator $[X,Y]$ is not in general a vector field but rather a second order differential operator. Second, every graph $G$ supports a notion of \textit{gradient} $\nabla,$ \textit{divergence}  $\Div,$ and \textit{Laplacian} $\Delta =\Div\circ\,\nabla.$ Because $\Delta\phi$ only involves differences of $\phi$ along directed edges with a given base point it is a first order operator in this theory. Since the tangent graph is itself a finite simple graph it has a Laplacian, denoted $\Delta_{tG},$ and we find that,
$$
L\phi = \Div\circ\,\Delta_{tG}\circ\nabla\phi
$$
is a natural second order differential operator, an explicit expression of which was derived in \cite{M1}.

\mpar
The observation that commutators of vector fields aren't necessarily vector fields is our jumping off point - what causes the failure of the second derivatives in the commutator to cancel? The cause is the lack of local spatial homogeneity in a graph and this, in turn, is a result of the discreteness of space.

\mpar
The discrete nature of a graph allows for significant local variability of space compared to the continuum. For example, the unit sphere $S(i, 1)$ at vertex $i$ can be an arbitrary graph on $\Deg(i)$ vertices and it can vary significantly from vertex to vertex. (See Remark 3.6.3). In contrast, there is only one choice for the unit sphere $S(x,1)$ at a point $x\in\mathbb{R}^d$ and it is the same choice at every other point. In contrast, while the unit ball at vertex $i$ is structured, being a star graph $K_{1,d}$ where $d$ is the degree of vertex $i$, the ball of radius two at $i$ is unstructured, since it can contain a subgraph isomorphic to any graph on $d$ vertices. Since $\Deg(i)$ and $S(i,1)$ may vary significantly from vertex to vertex, the generic graph is as locally inhomogeneous as possible.
 
\mpar
One way around this structural variability is to introduce $G^2,$  the \textit{second power graph} of $G,$ namely $G$ augmented by new edges between vertices that are distance two apart. We'll see in Proposition 3.3 that $DOp^2(G)\cong\mathcal{X}(G^2), $ so it's basically a tautology that vector fields on the second power graph of $G$ faithfully represent second order operators on $G.$ Thus, one could adopt the attitude that studying the geometry of $tG^2$ is equivalent to studying the structure of $DOp^2(G).$ However we strike a different attitude, namely that the tangent graph $tG$ has a much richer geometry than $G^2$ so there is more understanding to be gained by placing it at the center of our attention. 

\mpar
The key insight is this. On one hand, every pair of vertices of $G$ at distance two apart is joined by a least one pair of incident edges. All such pairs of incident edges in $G$ determine the same edge in $G^2.$  On the other hand, each pair of incident edges in $G$ determines a rectangle in the tangent graph $tG,$ and the rectangles are distinct from pair to pair. Thus, $G^2$ shows \textit{that} vertices of $G$ distance two apart are connected, but $tG$ shows precisely \textit{how} they are connected. This fact focuses attention on the geometry of balls of radius two and their dipole subgraphs, meaning the union of $P_2$ subgraphs whose endpoints are a fixed pair of vertices at distance two apart.

\mpar
It's natural to ask about the relationship between $G^2$ and $tG.$ Observe that in an obvious sense $G^2$ is smaller than $tG,$ since a single edge in $G^2$ may correspond to multiple rectangles in $tG.$  However, there are no injective graph homomorphisms from $G^2$ to $tG$ or surjective graph homomorphisms from $tG$ to $G^2,$ in general, as simple examples show. Similarly for injections from $tG^2$ to $t(tG),$ and surjections from $t(tG)$ to $tG^2.$ Nevertheless, there is a natural surjection from $\mathcal{X}(T^2G)$ to $DOp^2(G)$ although it is not functorial, meaning it doesn't follow from a homomorphism of the underlying graphs. Here $\mathcal{X}(T^2G)$ is the space of sections of \textit{second tangent bundle} of $G,$ 
$$
T^2_i(G)\cong \bigoplus_{\pi(u)=i}T_u(tG),
$$
the bundle whose fiber at vertex $i\in V_G$ is the vector space spanned by the indicator functions of vertices in $t^2G$ whose base point is $i.$ We know that $\mathcal{X}(T^2G)\cong\mathcal{X}(tG)$ so sections of the second tangent bundle, that is \textit{second order vetor fields,}  are in one-to-one correspondence with vector fields on the tangent graph $tG.$ 

\begin{definition}
The map $a\colon\mathcal{X}(T^2G)\to DOp^2(G)$ is defined by the rule,
$$
a(Z)\phi  =\sum_{\alpha\in V_{t^2G}}Z(\alpha)d^2\phi(\alpha)e_{\pi^2(\alpha)}.
$$
\end{definition}
\begin{remark}
\textit{In this formula, the vertex  $\alpha\in V_{t^2G}$ is a concatenation $\alpha=uv$ of vertices in an edge $\{u,v\}\in E_{tG},$ each of which is a concatenation of vertices $u=ij, v=kl$ of edges $\{i,j\}, \{l,k\}\in E_G.$ Recall that $ij/kl\in V_{tG}$ if and only if either $j=k$ or $i=l.$ Then $d^2\phi(\alpha)$ is the double difference,
\begin{align*}
d^2\phi(\alpha) & =d\phi(v)-d\phi(u)\\
& = (\phi(l)-\phi(k))-(\phi(j)-\phi(i)),
\end{align*}
and $\pi^2(\alpha) = \pi(u) = i$ is the base point of $\alpha$. The function $e_i$ is the indicator function of vertex $i,$ so $e_{\pi^2(\alpha)}$ is the indicator of the base point of $\alpha.$}
\end{remark}

\mpar
We'll see that $a$ is surjective and has a non-empty kernel except in trivial cases.  Observe that any subspace $F$ of $\mathcal{X}(T^2(G))$ such that $a\colon F\to DOp^2(G)$ is an isomorphism must satisfy $F\oplus \Ker(a) =\mathcal{X}(T^2(G))$. Thus, a canonical form is understood to be a geometrically meaningful injective map $c\colon DOp^2(G)\to\mathcal{X}(T^2(G))$ such that $a\negthinspace\circ\negthinspace c$ is the identity. Of course, the qualifying phrase \textit{geometrically meaningful} is subjective and we could just say that a canonical form is a choice of a subspace complimentary to $\Ker(a).$  But we insist on it because the salient point of a canonical form is the rule by which $DOp^2(G)$ is associated with a complementary subspace of $\Ker(a)$, not just the complementarity of the subspace itself.

\mpar
As mentioned previously, Section 2 of the paper is devoted to two examples of second order differential operators and general observations that follow from them. We'll see the simplest case of $DOp^2(P_2)$ where $P_2$ is a path of length two, is both interesting and informative. Working out these examples in detail helps fix ideas and familiarize notation. 

\mpar
In Section 3 we present a canonical form for $DOp^2(G)$ based on the notion of dipole subgraphs. We'll see that proper dipoles of $G$ are in one-to-one correspondence with edges in $E_{G^2}\setminus E_G$ so that dipoles record precisely how the vertices of an edge in $G^2$ are connected.

\mpar
In Section 4 we calculate the dimension of $DOp^2(G)$ in terms of the cyclomatic numbers of the balls of radius two in $G$, derive tight bounds on the dimension, and identify graphs where equality holds. 

\mpar
Recall that $\mathcal{X}(T^2(G))$ has a natural inner product. In Section 5 we characterize $\ker(a)$ as the kernel of a divergence-type operator $\DIV,$ specifically,
$$
a(Z)\phi(i) =\sum_{j\in V_G}\DIV Z_i(j)\phi(j)
$$
and derive a formula for $\DIV Z_i(j)$ explicitly in terms of $Z.$ An immediate consequence is a necessary and sufficient condition on $Z$ such that $a(Z)$ is self adjoint. We identify the adjoint of $\DIV$ as a Hessian-type operator $\Hess$ leading to a Helmholtz decomposition of $\mathcal{X}(T^2(G)).$ A second second canonical form follows from this decomposition. While the orthogonal projection onto $\ker(a)^\perp = \Image(\Hess)$ is highly non-local, we calculate it asymptotically via the method of gradient descent.

\mpar
In Section 6 we examine products of vector fields $XY\phi = X(Y\phi).$ The main result is an explicit formula for a second order vector field $X\circ Y\in\mathcal{X}(T^2(G))$ such that $a(X\circ Y)=XY.$ We use this formula to characterize pairs of vector fields such that their commutator is also a vector field. We also calculate a formula for the adjoint $(XY)^*.$

\mpar
In section 7 we look at operators generalizing the canonical second order differential operator $\Div\circ\,\Delta_{tG}\circ \nabla$ that was discussed in Section 2. We introduce the map $b\colon\mathcal{X}(tG)\to DOp^2(G)$ given by the formula,
$$
b(Z)=\Div\circ\, Z\circ\nabla,
$$
relate it to the map $a,$ and calculate its adjoint $b(Z)^*$ in terms of $Z.$

\mpar
Finally, in Section 8 we calculate the adjoint of $a(Z)$ as,
$$
a(Z)^*\phi(i)=a(W)\phi(i)+X\phi(i)+\nu(i)\phi(i)
$$
where $W\in\mathcal{X}(T^2(G)),\, X\in\mathcal{X}(G),$ and $\nu\in C(G)$ are explicit functions of $Z.$ This formula involves a complimentary notion of gradient and divergence on $tG$ which we call the \textit{acclivity} $\bigtriangledown_{tG}$ and the \textit{dispersion} $\Dsp_{tG}$, respectively, that may be of independent interest.

\mpar
A skeptic could reasonably ask: since second order differential operators are just band-limited matrices with entries indexed by $V_G\times V_G,$ why not deal with them directly and leave tangent graphs out of it? The answer rests on understanding what band limited means in the context of linear operators on general graphs. If $L\in DOp^2(G)$ then $L(i,j)=0$ if $d(i,j)>2,$ meaning it could be non-zero on $B(i,2),$ the ball of radius two centered at $i.$ As noted above, $B(i,2)$ can be wildly complicated for fixed $i$ and can differ significantly from vertex to vertex in the same graph. All the complications stemming from the lack of local spatial homogeneity in $G$ are summarized precisely by the second tangent graph $t^2G$ and  transferred faithfully to derived objects like $\mathcal{X}(tG)$ and $\mathcal{X}(T^2G).$ Generally speaking, there are many more paths of length two in $B(i,2)$ starting at $i$ than there are vertices in the ball, so there are bound to be redundancies in the map $a\colon \mathcal{X}(T^2G)\to DOp^2(G),$ redundancies which are captured in $\Ker(a).$ At the end of the day, this is nothing more than a covering argument: to understand a base object $B$, it's often convenient to introduce a covering $E\to B$ from a linearized object $E$ to $B$ and study its mapping properties.

\mpar
In the sequel we restrict attention to connected finite simple graphs. All our results extend in a natural way to general finite simple graphs by looking at connected components separately.
 
 \section{Two Examples}
At first glance, the notation of tangent graphs may seem overly elaborate but its benefit is to provide systematic and effective language for expressing differences of functions across directed edges of a graph $G$. When applied to the tangent graph $tG$ this language expresses double differences of functions across linked pairs of directed edges in $G$, etc. In particular, classical notions like gradient, divergence and Laplacian are naturally expressed in this language.

\mpar
Calculations like those in the first example are made throughout this work so it's helpful to dwell on them a little bit to gain some facility. This requires introducing a few ideas and certain amount of notation.

\mpar
Let $e_i,\, e_u$ and $e_\alpha$ be the indicator functions of vertices $i\in V_G,\, u\in V_{tG},$ and $\alpha\in V_{t^G},$ respectively. Then,
$$
C(G)=\langle e_i\mid i\in V_G\rangle, \,\,  C(tG)=\langle e_u\mid u\in V_{tG}\rangle, \,\, \text{and}\,\,C(t^2G)=\langle e_\alpha\mid u\in V_{t^2G}\rangle
$$ 
are the spaces of functions on the respective graphs.

\mpar
Let's look more closely at the relationship among vertices in $V_G, V_{tG}$ and $V_{t^2G}.$ (See Figure 1). Let $\alpha\in V_{t^2G},\,\, u=\pi(\alpha), v=\pi_+(\alpha),$ and let,
$$
i=\pi(u),\,\, j=\pi_+(u),\,\,k=\pi(v),\,\,\text{and}\,\, l=\pi_+(v).
$$
Then $\{u,v\}\in E_{tG},\,\,\{i,j\}, \{k,l\}\in E_G$ and we write $u=ij, v=kl$ and,
$$
\alpha=uv=ij/kl
$$
By definition either $j=k$ or $i=l.$ In the first case $\alpha=ij/jk$ and in the second case $\alpha=ij/ki.$ It's useful to think of $ij/jk$ as a forward translation of $ij$ to $jk$ and $ij/ki$ as a backward translation of $ij$ to $ki$. When $j=k$ and $i=l$ we think of $ij/ji$ as a reflection sending $ij$ to $ji.$ We say that an edge in $tG$ of the form $\{ij,ji\}$ is called \textit{central} and all other edges are called \textit{transverse.}

\mpar
\begin{figure}[h]
\begin{tikzpicture}
\draw[fill=black] (0,0) circle (2pt);
\draw[fill=black] (1,0) circle (2pt);
\draw[fill=black] (2,0) circle (2pt);
\node at (0,-0.30) {\textit{i}};
\node at (1,-0.3) {\textit{j}};
\node at (2,-0.3) {\textit{k}};
\draw[thick] (0,0) --( 1,0) -- (2, 0);

\draw[fill=black] (4,0) circle (2pt);
\draw[fill=black] (5,0) circle (2pt);
\node at (4,-0.3) {\textit{i}};
\node at (5,-0.3) {\textit{j}};
\draw[thick] (4,0) --( 5,0);
\draw[thick,->] (4,0) -- (4.6,0);

\draw[fill=black] (6,0) circle (2pt);
\draw[fill=black] (7,0) circle (2pt);
\node at (6,-0.3) {\textit{i}};
\node at (7,-0.3) {\textit{j}};
\draw[thick] (6,0) --( 7,0);
\draw[thick,-<] (6,0) -- (6.5,0);

\draw[fill=black] (8,0) circle (2pt);
\draw[fill=black] (9,0) circle (2pt);
\node at (8,-0.3) {\textit{j}};
\node at (9,-0.3) {\textit{k}};
\draw[thick] (8,0) --(9,0);
\draw[thick,->] (8,0) -- (8.6,0);

\draw[fill=black] (10,0) circle (2pt);
\draw[fill=black] (11,0) circle (2pt);
\node at (10,-0.3) {\textit{j}};
\node at (11,-0.3) {\textit{k}};
\draw[thick] (10,0) --( 11,0);
\draw[thick,-<] (10,0) -- (10.5,0);

\draw[fill=black] (0,-1) circle (2pt);
\draw[fill=black] (1,-1) circle (2pt);
\draw[fill=black] (2,-1) circle (2pt);
\node at (0,-1.30) {\textit{i}};
\node at (1,-1.3) {\textit{j}};
\node at (2,-1.3) {\textit{k}};
\draw[thick] (0,-1) --( 1,-1) -- (2, -1);
\draw[thick, ->] (1,-1) -- (1.6, -1);
\draw[thick, -Latex] (0,-1) -- (0.7, -1);

\draw[fill=black] (3,-1) circle (2pt);
\draw[fill=black] (4,-1) circle (2pt);
\draw[fill=black] (5,-1) circle (2pt);
\node at (3,-1.30) {\textit{i}};
\node at (4,-1.3) {\textit{j}};
\node at (5,-1.3) {\textit{k}};
\draw[thick] (3,-1) --( 4,-1) -- (5, -1);
\draw[thick, -Latex] (5.0,-1) -- (4.4, -1);
\draw[thick, -<] (3,-1) -- (3.6, -1);

\draw[fill=black] (6,-1) circle (2pt);
\draw[fill=black] (7,-1) circle (2pt);
\draw[fill=black] (8,-1) circle (2pt);
\node at (6,-1.30) {\textit{i}};
\node at (7,-1.3) {\textit{j}};
\node at (8,-1.3) {\textit{k}};
\draw[thick] (6,-1) --( 7,-1) -- (8, -1);
\draw[thick, ->] (6.0,-1) -- (6.6, -1);
\draw[thick, -Latex] (7,-1) -- (7.7, -1);

\draw[fill=black] (9,-1) circle (2pt);
\draw[fill=black] (10,-1) circle (2pt);
\draw[fill=black] (11,-1) circle (2pt);
\node at (9,-1.30) {\textit{i}};
\node at (10,-1.3) {\textit{j}};
\node at (11,-1.3) {\textit{k}};
\draw[thick] (9,-1) --( 10,-1) -- (11, -1);
\draw[thick, -Latex] (10,-1) -- (9.3, -1);
\draw[thick, -<] (10,-1) -- (10.7, -1);

\draw[fill=black] (2,-2) circle (2pt);
\draw[fill=black] (3,-2) circle (2pt);
\node at (2,-2.3) {\textit{i}};
\node at (3,-2.3) {\textit{j}};
\draw[thick] (2,-2) -- (3,-2);
\draw[thick, -<] (2, -2) -- (2.6, -2);
\draw[thick, Latex-](2.6,-2)--(2,-2);

\draw[fill=black] (4,-2) circle (2pt);
\draw[fill=black] (5,-2) circle (2pt);
\node at (4,-2.3) {\textit{i}};
\node at (5,-2.3) {\textit{j}};
\draw[thick] (4,-2) -- (5,-2);
\draw[thick, ->] (4, -2) -- (4.5, -2);
\draw[thick, -Latex] (5,-2)--(4.4,-2);

\draw[fill=black] (6,-2) circle (2pt);
\draw[fill=black] (7,-2) circle (2pt);
\node at (6,-2.3) {\textit{j}};
\node at (7,-2.3) {\textit{k}};
\draw[thick] (6,-2) -- (7,-2);
\draw[thick, -<] (6, -2) -- (6.6, -2);
\draw[thick, Latex-](6.6,-2)--(6,-2);

\draw[fill=black] (8,-2) circle (2pt);
\draw[fill=black] (9,-2) circle (2pt);
\node at (8,-2.3) {\textit{j}};
\node at (9,-2.3) {\textit{k}};
\draw[thick] (8,-2) -- (9,-2);
\draw[thick, ->] (8, -2) -- (8.5, -2);
\draw[thick, -Latex] (9,-2)--(8.4,-2);

\end{tikzpicture}
\caption{\textit{Top Row:} Two incident edges in $G$ and their four directed counterparts \textit{(l-r):} $ij,\,ji,,jk,\,kj$, which are vertices in $tG.$ The four vertices combine to make four edges in $tG$, two central: $\{ij, ji\}, \{jk,kj\}$ and two transverse: $\{ij, jk\}, \{kj, ji\}.$ These four edges in $tG$ combine to make eight vertices $t^2G$.  For vertices $\alpha\in t^2G,$ a solid arrowhead indicates the base point $\pi(\alpha)$ and an open arrowhead indicates the end point $\pi_+(\alpha)$. \textit{Middle Row (l-r):} There are two forward translations: $ij/jk,\, kj/ji$ and two backward translations: $jk/ij,\, ji/kj.$ \textit{Bottom row (l-r):} There are four reflections: $ij/ji,\, ji/ij,\, jk/kj,\, kj/jk.$}
\end{figure}
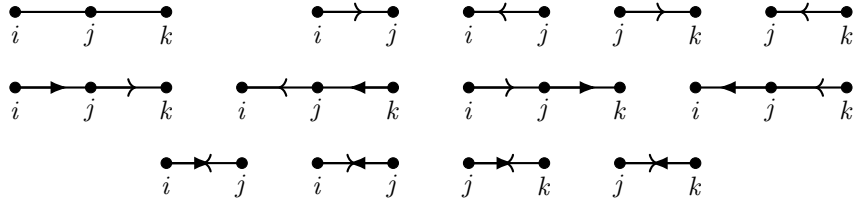

Writing vertices in a tangent graph as concatenations of vertices in an edge of the given graph is useful notation and it can be iterated, as we've just seen. The projections $\pi$ and $\pi_+$ can be iterated as well, for example,
$$
\pi^2(\alpha)=\pi(\pi(\alpha))=\pi(u)=i, \,\,\text{etc.}
$$

\mpar
Observe that a function $X\in C(tG)$ determines a first order differential operator by the rule,
$$
X\phi=\sum_{u\in V_tG}X(u)d\phi(u)e_{\pi(u)}.
$$
So, if $X=e_{ij}$ then,
$$
X\phi = d\phi(ij)e_i = (\phi(j)-\phi(i))e_i.
$$
In addition, if $\phi=e_k$ then,
\begin{align*}
X\phi & = e_{ij}\, e_k\\
& =(e_k(j)-e_k(i))e_i\\
& = \begin{cases}
-e_i, &\text{$k=i,$}\\
\phantom{-} e_i, &\text{$k=j,$}\\
\phantom{-} 0, &\text{$k\neq i,j.$}
\end{cases}
\end{align*}

\mpar
Written differently we have,
$$
e_u e_i = d e_i(u) e_{\pi(u)} =(e_i(\pi_+(u))-e_i(\pi(u))) e_{\pi(u)}
$$
for all $u\in V_{tG}$ and $i\in V_G.$

\mpar
Similarly, a function $Z\in C(t^2G)$ determines a second order differential operator by the rule,
$$
a(Z)\phi  =\sum_{\alpha\in V_{t^2G}}Z(\alpha)d^2\phi(\alpha)e_{\pi^2(\alpha)}.
$$
We could have written this simply as $Z\phi,$ as we did with $X\phi,$ but that would be ambiguous since we reserve the adjacency notation $Zf$ to mean the action of the vector field $Z$ on the function $f\in C(tG).$ Note that $C(tG)\cong\mathcal{X}(G)$ so $Z$ acts on vector fields while $a(Z)$ acts on functions. This notation reflects the distinction between the isomorphic vector spaces $\mathcal{X}(tG)$ and $\mathcal{X}(T^2G).$

\begin{example}
Let $G$ be $P_2$, the path of length two with vertices $\{1,2,3\}$. It's a pleasant exercise to check that  $P_2^2$ is a triangle, $tP_2$ is a rectangle, $tP_2^2$ is a triangular wedge, and $t^2P_2$ is a cube. (See Figure 2). The goal is to show that second order differential operators on $P_2$ can be described in multiple ways: as vector fields on $P_2^2$, as products of vector fields on $P_2,$ as second order vector fields on $P_2$, and as so-called \textit{$2$-charge fields}. We also show that the commutator of vector fields on $P_2$ is generally a second order differential operator, not a vector field.

\begin{figure} 
\begin{tikzpicture}
\draw[fill=black] (0,0) circle (2pt);
\draw[fill=black] (1.5,0) circle (2pt);
\draw[fill=black] (0.75,1.5) circle (2pt);

\node at (0,-0.30) {1};
\node at (1.5,-0.3) {2};
\node at (0.75,1.8) {3};

\draw[thick] (0,0) -- (1.5,0) -- (0.75,1.5) -- (0,0);

\draw[fill=black] (2.5,0) circle (2pt);
\draw[fill=black] (4,0) circle (2pt);
\draw[fill=black] (5.5,0) circle (2pt);

\node at (2.5,-0.30) {1};
\node at (4,-0.3) {2};
\node at (5.5,-0.3) {3};

\draw[thick] (2.5,0) --( 4,0) -- (5.5, 0);

\draw[fill=black] (6.5,0) circle (2pt);
\draw[fill=black] (6.5,1.5) circle (2pt);
\draw[fill=black] (8,0) circle (2pt);
\draw[fill=black] (8,1.5) circle (2pt);

\node at (6.5,-0.3) {12};
\node at (6.5,1.8) {21};
\node at (8,-0.3) {23};
\node at (8,1.8) {32};

\draw[thick] (6.5,0) -- (8,0) -- (8,1.5) -- (6.5, 1.5) -- (6.5,0);

\draw[fill=black] (0.5,-4) circle (2pt);
\draw[fill=black] (3,-4) circle (2pt);
\draw[fill=black] (2,-3.5) circle (2pt);
\draw[fill=black] (0.5,-2.5) circle (2pt);
\draw[fill=black] (3,-2.5) circle (2pt);
\draw[fill=black] (2,-2) circle (2pt);

\node at (2, -1.7) {13};
\node at (2,-3.8) {31};
\node at (0.5,-2.2) {21};
\node at (0.5, -4.3) {12};
\node at (3,-2.2) {32};
\node at (3,-4.3) {23};

\draw[thick] (0.5,-4) -- (3,-4) -- (3, -2.5) -- (0.5,-2.5) -- (0.5, -4);
\draw[thick] (0.5, -4) -- (2,-3.5) -- (3, -4);
\draw[thick] (0.5, -2.5) -- (2, -2) -- (3, -2.5);
\draw[thick] (2,-3.5) -- (2, -2);

\draw[fill=black] (4.5,-4) circle (2pt);
\draw[fill=black] (6.5,-4) circle (2pt);
\draw[fill=black] (4.5,-2.5) circle (2pt);
\draw[fill=black] (6.5,-2.5) circle (2pt);
\draw[fill=black] (5.5,-3.5) circle (2pt);
\draw[fill=black] (7.5,-3.5) circle (2pt);
\draw[fill=black] (5.5,-2) circle (2pt);
\draw[fill=black] (7.5,-2) circle (2pt);

\node at (4.5,-4.3) {\textit{a}};
\node at (6.5,-4.3) {\textit{b}};
\node at (4.5,-2.2) {\textit{e}};
\node at (6.5,-2.25) {\textit{f}};
\node at (5.5,-3.8) {\textit{d}};
\node at (7.5, -3.8) {\textit{c}};
\node at (5.5, -1.7) {\textit{h}};
\node at (7.5, -1.7) {\textit{g}};

\draw[thick] (4.5,-4) -- (6.5,-4) -- (6.5, -2.5) -- (4.5, -2.5) -- (4.5,-4);
\draw[thick] (5.5,-3.5) -- (7.5,-3.5) -- (7.5, -2) -- (5.5, -2) -- (5.5,-3.5);
\draw[thick] (4.5, -4) -- (5.5,-3.5);
\draw[thick] (4.5, -2.5) -- (5.5,-2);
\draw[thick] (6.5, -4) -- (7.5,-3.5);
\draw[thick] (6.5, -2.5) -- (7.5 ,-2);
\end{tikzpicture}
\caption{\textit{Top row:} The graph $P_2$ \textit{(center)} flanked by its second power graph, $P_2^2$ \textit{(left)} and its tangent graph, $tP_2$ \textit{(right).} \textit{Bottom row:} $tP_2^2,$ the tangent graph of the second power graph \textit{(left)}, and  $t^2P_2=t(tP_2),$ the second tangent graph of $P_2$ \textit{(right).} The vertices of $t^2P_2$ are labelled $a=12/23,$ $b=23/32,$ $c=32/21$, $d=21/12,$ $e=23/12,$ $f=32/23,$ $g=21/32,$ $h=12/21.$}
\end{figure}
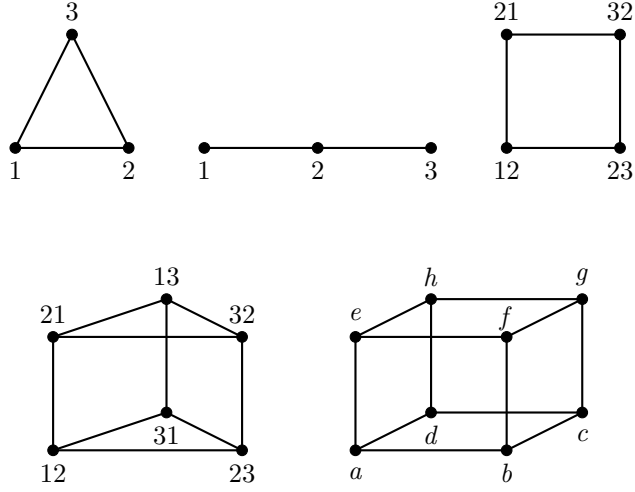

\mpar
First, note that $|DOp^1(P_2)| = 4,$ because $|\mathcal{X}(G)|=|V_{tG}|= 2|E_G|$ for any graph $G.$ Let's check that $|DOp^2(P_2)|=6.$ To see it, note that $DOp^1(G)\subset DOp^2(G)$ for any graph which accounts for 4 dimensions when $G=P_2.$ Again, for any graph it's clear that the operators,
$$
L^{ij}\phi=(\phi(j)-\phi(i))e_i
$$ 
are in $DOp^2(G)$ if and only if $d(i,j)=1,2,$  and they form a basis of $DOp^2(G).$ Thus $L^{13}, L^{31}$ are linearly independent operators in $DOp^2(P_2)\negthinspace\setminus\negthinspace DOp^1(P_2),$ and it follows that $|DOp^2(P_2)|=6.$ 

\mpar
Observe that $P_2^2$ has 3 edges and therefore $|DOp^1(P_2^2)|=6,$ meaning it has the same dimension as $DOp^2(P_2).$ It is not a coincidence that $DOp^2(P_2)\cong DOp^1(P_2^2)$. They have the same dimension because the new edge $\{1,3\}$ in $P_2^2$ accounts for two vertices $13$ and $31$ in $tP_2^2$ that correspond to vector fields $e_{13}, e_{31}\in \mathcal{X}(P_2^2)$ and to operators $L^{13}, L^{31}\in DOp^2(P_2).$ We'll see below that these spaces are isomorphic for every graph $G.$

\mpar
Second, let's show products of vector fields $(XY)\phi=X(Y\phi)$ are second order differential operators that span $DOp^2(P_2).$ It's easy to verify but it's worthwhile writing out the details. It suffices to check that $L^{12}$ and $L^{13}$ can be written as linear combinations of products of vector fields since the other cases are similar. Observe that,
$$
e_{12}(e_{12}\phi) = e_{12}\left(d\phi(12)e_1\right) = d\phi(12)e_{12}e_1 = -(\phi(2)-\phi(1))e_1 =-L^{12} e_1,
$$
and,
$$
e_{12}(e_{23})\phi = e_{12}\left(d\phi(23)e_2\right)= d\phi(23) e_{12}e_2 = (\phi(3)-\phi(2))e_1.
$$

Thus, 
$$
L^{12}\phi =-e_{12}(e_{12}\phi) \quad\text{and}\quad L^{13}\phi = e_{12}(e_{23}\phi) -e_{12}(e_{12}\phi).
$$
Note that,
$$
|\mathcal{X}(P_2)\times \mathcal{X}(P_2)| = 4\times 4 = 16,
$$
 but $DOp^2(G)$ has dimension six, so there are ten independent relations among the products $XY.$ Roughly speaking, these relations are accounted for by zero divisors $XY=0,$ such as $X=e_{23}$ and $Y=e_{12},$ and by stabilizers $XY=Y,$ such as $X=-e_{21}$ and $Y=e_{23}.$

\mpar
Third, it's reasonable to expect that functions on $t^2G$ should correspond to second order operators in the same sense that functions on $tG$ correspond to first order operators. For example consider the vertex $\alpha= 12/23$ of $t^2P_2$ and its associated operator $a(e_\alpha)$ namely,
$$
a(e_\alpha)\phi=d^2\phi(12/23)e_1 =[\phi(3)-2\phi(2) +\phi(1)]e_1.
$$
Since $t^2P_2$ has eight vertices there are eight such operators inside $DOp^2(P_2),$ which is a space of dimension six. It's not hard to see that they span $DOp^2(P_2)$ so there are two linearly independent relations among them. Let's check that they do span $DOp^2(P_2)$ and find the two relations, just for practice. It suffices to show that $L^{32}$ and $L^{31}$ can be written as linear combinations of operators $a(e_\alpha)$ for $\alpha\in V_{t^2P_2}$ because the other cases are similar. If $\alpha=32/23$ then 
$$
a(e_\alpha)\phi=-2(\phi(2)-\phi(3))e_3 = -2 L^{32}\phi.
$$
On the other hand, if $\beta=32/21$ then,
\begin{align*}
a(e_\beta)\phi & = (\phi(1)-2\phi(2)+\phi(3))e_3\\
& = ([\phi(1)-\phi(3)] -2[\phi(2)-\phi(3)]) e_3\\
 & =L^{31}\phi - 2L^{32}\phi\\
 & = L^{31}\phi - a(e_\alpha) \phi,
\end{align*}
hence $L^{31}=a(e_\alpha+e_\beta).$ 

\mpar
So, there is a natural surjection $a\colon\mathcal{X}(T^2P_2)\to DOp^2(P_2)$ from the space of second order vector fields to second order operators. It's not an isomorphism, as it is for first order differential operators and vector fields, so let's examine the kernel.
\mpar
To find a basis of $\Ker(a),$ note that,
$$
a(e_{21/32})\phi  = (2\phi(2)-\phi(1)-\phi(3))e_2  = a(e_{23/12})\phi,
$$
and it follows that  $Z_1=e_{21/32}-e_{23/12}\in\Ker(a).$ Next, observe that,
$$
a(e_{21/12})\phi = 2(\phi(2)-\phi(1))e_2\quad\text{and}\quad a(e_{23/32})\phi = 2(\phi(2)-\phi(3))e_2.
$$
Thus,
$$
a(e_{21/32})\phi - a(e_{23/32})\phi = (\phi(3)-\phi(1))e_2= a(e_{21/12})\phi - a(e_{23/12})\phi
$$
and therefore, $Z_2=  e_{21/32} + e_{23/12} -e_{21/12} - e_{23/32}\in\Ker(a),$ as well. Clearly, $Z_1$ and $Z_2$ are linearly independent so they form a basis of $\Ker(a).$

\mpar
At this point we have three descriptions of $DOp^2(P_2)$:  as vector fields on $P_2^2$, as products of vector fields on $P_2$, and as cosets $\mathcal{X}(T^2P_2)/\Ker(a).$ There is a fourth way of seeing $DOp^2(P_2)$, namely as a field of $2$-charges $L(1, \cdot), L(2, \cdot),$ and  $L(3, \cdot),$ where a $2$-charge based at vertex $i$ is a function $\zeta_i$ such that, $\zeta_i(j)=0$ if $d(i,j)>2$ and,
$$
\sum_{d(i,j)\leq 2}\zeta_i(j)=0.
$$
This description is important but it's basically a tautology: the fact that $L$ annihilates constants is equivalent to the condition that the row sums of $L$ vanish which is equivalent to the condition that the row vectors form a section of the 2-charge bundle. We'll see this four-way description of second order operators holds for any graph.  

\mpar
Next, let's exhibit a pair of vector fields whose commutator is not a vector field. Let $X=e_{12}$ and $Y= e_{23}.$ Then, $ X\phi=d\phi(12)e_1$ and $Y\phi=d\phi(23)e_2.$ One calculates $X(Y\phi)(1)=\phi(3)-\phi(2)$ and $Y(X\phi)(1)=0$ hence, 
$$
[X,Y]\phi(1) =\phi(3)-\phi(2).
$$
Evidently, $[X,Y]\phi(1)$ depends on the values of $\phi$ on the ball of radius two about vertex $1,$ so $[X ,Y]$ is not a first order operator hence not a vector field.

\mpar
To conclude this example, let's look at the relationship between the second power graph $P_2^2$ and the tangent graph $tP_2$ as well as the relationship between their tangent graphs $tP_2^2$ and $t^2P_2.$ The motivation for doing so is the hope that the surjection $a\colon \mathcal{X}(T^2P_2)\to DOp^2(P_2)$ is functorial in the sense that it derives from an underlying surjective graph homomorphism from $t^2P_2$ to $tP_2^2.$  This hope is in vain, however.

\mpar
It's easy to see that there are no graph homomorphisms from the triangle $P_2^2$ to the square $tP_2$ and the only morphisms from the square to the triangle are not surjective. This makes it seem unlikely that there is a surjective morphism from the cube $t^2P_2$ to the triangular wedge $tP_2^2$ but it's not so easy to prove. If there were such a morphism then $t^2P_2$ would have six independent sets corresponding to the six vertices of $tP^2_2$ and edges between at least one pair of vertices in distinct independent sets if their corresponding vertices were adjacent in $tP^2_2.$ Out of an abundance of curiosity we've enumerated the set of twelve graphs that can be formed in this way from six independent sets in $t^2P_2$ and observe the triangular wedge is not among them. The details are provided in Appendix B.
\end{example}

\begin{remark}
\textit{
1. In a complete graph every vertex is adjacent to every other vertex and therefore every differential operator is first order, including $[X,Y].$ On the other hand, removing an edge from a complete graph creates a $P_2$ subgraph whose endpoints, by definition, are not adjacent. By the calculations above, there exists a pair of vector fields on the residual graph whose commutator is not a vector field. Therefore, the property that every commutator of vector fields is a vector field characterizes complete graphs.}

\mpar
\textit{
2. This characterization suggests that the second order part of the commutator is a bilinear form on vector fields whose non-vanishing detects local departures from completeness -  a sort of coarse curvature operator. To make this precise, define the subspace of pure second order differential operators by the rule,
$$
DOp^2_+(G)=\langle L\in DOp^2(G) \mid L(i,j)=0,\,\, d(i,j)=1\rangle.
$$
}
\textit{Then,
$$
DOp^2(G)=DOp^1(G)\oplus DOp_+^2(G)
$$
and we write $L_+$ for the purely second order part of $L.$ Then the form $(X,Y)\to [X,Y]_+$ vanishes when $X$ and $Y$ are supported in the same clique of $G.$}
\end{remark}

\begin{example}
In our study of Bochner's identity \cite{M1} we encountered the second order differential operator  $\Div\circ \,\Delta_{tG}\circ \nabla$ and evaluated it explicitly,
$$
\Div\Delta_{tG}\nabla\phi(i) =\Delta^2\phi(i) -4\Delta\phi(i) - 4D\phi(i),
$$
where $\Delta$ is the \textit{Laplacian} of $G,\,\Delta_{tG}$ is the Laplacian of the tangent graph $tG,$ and $D$ is the vector field,
$$
D(u)=\Deg(\pi_+(u))+\Deg(\pi(u)).
$$

\mpar
Recall that the \textit{gradient} of a function is the vector field whose value at $u\in V_{tG}$ is,
$$
\nabla\phi(u)= d\phi(u)=\phi(\pi_+(u))-\phi(\pi(u)),
$$
and the \textit{divergence} of a vector field is the function whose value at $i\in V_G$ is
$$
\Div X(i) =\sum_{\pi_+(u)=i}\negthickspace X(u) \thickspace - \sum_{\pi(u)=i}\negthickspace X(u).
$$

\mpar 
These formulas define the gradient operator $\nabla\colon C(G)\to\mathcal{X}(G)$,
$$
\nabla\phi = \sum_{u\in V_{tG}} d\phi(u)e_u,
$$
and the divergence operator $\Div\colon\mathcal{X}(G)\to C(G)$,
$$
\Div X = \sum_{i\in V_G} \Div X(i) e_i.
$$

There is a natural inner product on the tangent space, 
$$
T_i(G)= \langle e_u\mid u\in V_{tG},\,\,\pi(u)=i\rangle
$$
in which the indicator functions of vertices are orthonormal,
$$
\langle e_u, e_v\rangle_{T_i(G)} = 
\begin{cases}
1, & u=v,\\
0, & u\neq v
\end{cases}
$$
that extends to an inner product on vector fields,
$$
\langle X,Y\rangle_{\mathcal{X}(G)} =\sum_{u\in V_{tG}}X(u)Y(u).
$$ 

The natural inner product on $C(G)$ is,
$$
\langle \phi, \psi\rangle_{C(G)}=\sum_{i\in V_G}\phi(i)\psi(i)
$$
and it's known that the gradient and divergence are mutually adjoint with respect to these inner products,
$$
\langle \nabla\phi, X\rangle_{\mathcal{X}(G)} = \langle \phi, \Div X\rangle_{C(G)}.
$$

The Laplacian on $G$ is $\Delta = \Div\circ\nabla.$ It is self adjoint since,
$$
\langle\Delta\phi, \psi\rangle_{C(G)}=\langle\nabla\phi,\nabla\psi\rangle_{\mathcal{X}(G)} = \langle\phi,\Delta\psi\rangle_{C(G)}
$$
and an easy calculation shows,
$$
\Delta\phi(i)=-2\negthickspace\sum_{\pi(u)=i}d\phi(u).
$$

It seems counterintuitive at first but the Laplacian is a first order differential operator in this theory, given by the constant vector field $\Delta(u) =-2.$

\mpar
Since $tG$ is a finite simple graph it has its own gradient, divergence, and Laplacian that we denote with subscripts. Since $\mathcal{X}(G)\cong C(tG)$ we think of the Laplacian of $tG$ as acting on vector fields on $G$, that is $\Delta_{tG}\colon\mathcal{X}(G)\to\mathcal{X}(G)$ where,
$$
\Delta_{tG}\,X(u) = \Div_{tG}\nabla_{tG}\,X(u) =-2\negthickspace\sum_{\pi(\alpha)=u}dX(\alpha).
$$ 
In particular, when $X=\nabla\phi$ is a gradient we have,
\begin{align*}
\Delta_{tG}\nabla\phi(u)& = -2\negthickspace\sum_{\pi(\alpha)=u} d\,\nabla\phi(\alpha)\\
& =-2\negthickspace\sum_{\pi(\alpha)=u}[d\phi(\pi_+(\alpha))-d\phi(\pi(\alpha))]\\
& = -2\negthickspace\sum_{\pi(\alpha)=u}d^2\phi(\alpha)
\end{align*}

where,
$$
d^2\phi(\alpha)=\phi(\pi^2_+(\alpha))-\phi(\pi(\pi_+(\alpha)))-\phi(\pi_+(\pi(\alpha)))+\phi(\pi^2(\alpha)).
$$
Observe that $d^2\phi(\alpha)$ only involves terms of the form $\phi(i)\pm\phi(j)$ where $d(i,j)\leq 2.$ Since the formula for $\Div\Delta_{tG}\nabla\phi(i)$ consists of sums of such terms indexed by vertices $\alpha\in V_{t^2G}$ such that $\pi^2(\alpha)=i$ or $\pi^2_+(\alpha)=i$, it follows that $\Div\Delta_{tG}\nabla$ is a second order differential operator. Evaluating the divergence of this vector field exactly takes a bit more work but no new ideas.

\mpar
Let's conclude this example by observing that,
$$
\Div\Delta_{tG}\nabla = \Div\circ(\Div_{tG}\nabla_{tG})\circ\nabla =(\Div\Div_{tG})\circ(\nabla_{tG}\nabla)
$$
is self adjoint. Evidently the iterated gradient $\nabla_{tG}\nabla\phi(\alpha)$ is the Hessian of the function $\phi$ and the iterated divergence $\Div\Div_{tG} Z(i)$ is its adjoint. These ideas play a role in formulating the second canonical form of $DOp^2(G)$ in Section 5.

\begin{remark}
\textit{
It's not a stretch to imagine substituting a general vector field $Z$ in place of $\Delta_{tG}$, leading to differential operators of the form $b(Z)\phi=\Div\circ \,Z\circ\nabla\phi.$ In Section 7 we show that $\Image(b)$ is a proper subspace of $DOp^2(G)$, so the map $b$ is a less general representation than the map $a.$ Since $a$ is surjective, for every $Z$ there is a vector field $W$ such that $b(Z)=a(W).$ We find a formula for one such $W$ and we also find a formula for the adjoint $b(Z)^*$, both of which are expressed directly in terms of $Z.$
}
\end{remark}
\end{example}

\section{First Canonical Form via the Dipole Decomposition}
We begin this section by showing the fourfold representation of $DOp^2(P_2)$ described in Example 2.1 holds for any graph $G.$ First, we recall the definition of a pair of \textit{vector bundles.}

\begin{definition}
1. A \textit{k-local function} at vertex $i$ is a function $\phi\colon V_G\to\mathbb{R}$ such that $\phi(j)=0,$ if $d(i,j)> k.$ Let $Lf_i^k(G)$ be the vector space of $k$- local functions at vertex $i$ and let, 
$$
Lf^k(G)=\coprod_{i\in V_G} Lf^k_i(G)
$$
be the vector bundle of $k$-local functions. The space of sections of the $k$-local function bundle is denoted $\mathcal{X}(Lf^k(G)).$

\mpar
2. A \textit{k-charge} at vertex $i$ is a $k$-local function at vertex $i$ such that  $\sum_{j\in V_G}\phi(j)=0.$ Let $Ch_i^k(G)$ be the vector space of $k$-charges at vertex $i$ and let, 
$$
Ch^k(G)=\coprod_{i\in V_G} Ch^k_i(G)
$$
be the vector bundle of $k$-charges. The space of sections of the 2-charge bundle is denoted $\mathcal{X}(Ch^2(G)).$
\end{definition}

\begin{remark}
\textit{
1. Note that $Ch^k(G)$ is a sub-bundle of $Lf^k(G)$ and the space of sections of the $k$-charge bundle $\mathcal{X}(Ch^k(G))$ is a subspace of the sections of the $k$-local function bundle $\mathcal{X}(Lf^k(G)).$
}
\mpar
\textit{2. Note that $\mathcal{X}(Lf^2(G))\cong Op^2(G),$ the set of second order operators on $G$, and $\mathcal{X}(Ch^2(G))\cong DOp^2(G).$}
\end{remark}
\begin{proposition}
For any finite simple graph $G,$
$$
DOp^2(G)=\{XY|X, Y\in\mathcal{X}(G)\}\cong\mathcal{X}(G^2)\cong\mathcal{X}(Ch^2G)\cong \mathcal{X}(T^2G)/\Ker(a).
$$
\end{proposition}

\begin{proof}
This is an elementary result but writing out the details is instructive.
\mpar
It's easy to see that $DOp^2(G)$ has a privileged basis consisting of operators of the form $L^{ij}\phi = (\phi(j)-\phi(i))e_i,$ where $d(i,j)=1,2.$ These are in one-to-one correspondence with vectors $e_{ij}\in T_i(G^2)$ hence $DOp^2(G^2)\cong \mathcal{X}(G^2).$ 

\mpar
If $L\in DOp^2(G)$ then $\zeta_i(j) = L(i,j)$ is a $2$-charge based at $i,$ because $L1=0$ implies $\sum_{j\in V_G}L(i,j)=0.$ Thus, $L$ determines a section of $\mathcal{X}(Ch^2G).$ Conversely any section $\zeta$ of the $2$-charge bundle determines a second order differential operator by same the rule, hence $DOp^2(G)\cong\mathcal{X}(Ch^2G).$

\mpar
Let $\{i,j\}, \{j,k\}\in E_G$ be incident edges. Then $ij, jk\in V_{tG}$ and $X=e_{ij}, Y=e_{jk}\in \mathcal{X}(G).$ There are two alternatives: either$\{i,j,k\}$ is or is not a triangle.  If not then, 
\begin{align*}
XY\phi & = X(Y\phi) =e_{ij}(e_{jk}\phi)\\ 
& = e_{ij}(d\phi(jk) e_j)\\
& = d\phi(jk)(e_j(j)-e_j(i))e_i\\
& = (\phi(k)-\phi(j)) e_i.
\end{align*}

This is a second order differential operator because $d(i,k)=2.$ On the other hand, if $\{i,j,k\}$ is a triangle then.
$$
XY\phi = (\phi(k)-\phi(j)) e_i = ([\phi(k)-\phi(i)] - [\phi(j)-\phi(i)]) e_i = (e_{ik}-e_{ij})\phi,
$$
hence $XY$ is a vector field.  Finally, if $\{i,j\}, \{k, l\}$ are not incident edges and $X=e_{ij}$ and $Y=e_{kl}$ then $XY=0.$ Specifically,
$$
XY\phi =e_{ij}(e_{kl}\phi) =e_{ij}(d\phi(kl)e_k) = d\phi(kl)(e_k(j)-e_k(i)) = 0,
$$
since $k\notin\{i,j\}.$ By linearity, this shows $\{XY\mid X,Y\in\mathcal{X}(G)\}\subset DOp^2(G).$

\mpar
To show the reverse inclusion it's enough to find vector fields $X,Y$ such that $XY=L^{ij}$ for all $d(i,j)=1,2.$ If $d(i,j)=1$ then choose $X=-Y=e_{ij}$. If $d(i,j)=2$ then there is a vertex $k\in V_G$ such that $\{i,k\}, \{k,j\}\in E_G$ and we choose $X=-e_{ik}$ and $Y=e_{kj}+e_{ik}$. Then calculations similar to the ones above show $XY\phi=L^{ij}\phi.$

\mpar
Finally, let's show that $a\colon\mathcal{X}(T^2G)\to DOp^2(G)$ is surjective which would imply  $DOp^2(G)\cong\mathcal{X}(T^2G)/\Ker(a).$
The set $\{ L^{ij}\mid \{i,j\}\in E_{G^2}\}$ is a basis of $DOp^2(G)$ so it's enough to show that for every such edge $\{i,j\}$ there exists a section $Z\in\mathcal{X}(T^2G)$ such that $a(Z)=L^{ij}.$ 

\mpar
Suppose first that $d(i,j)=1.$ Then,
$$
a(e_{ij/ji})\phi =d^2\phi(ij/ji)e_i = -2(\phi(j)-\phi(i))e_i.
$$
Choosing  $Z=-\tfrac{1}{2}e_{ij/ji},$ we find $a(Z)=L^{ij}.$ 

\mpar
Now suppose $d(i,j)=2.$ Then there exists a vertex $k\in V_G$ such that $\{i,k\}, \{k,j\}\in V_G,$ hence $ik/kj\in V_{t^2G}.$ Note that,
\begin{align*}
a(e_{ik/kj})\phi & = d^2\phi(ik/kj)e_i\\
& = (\phi(j)-2\phi(k) + \phi(i))e_i \\
& =([\phi(j)-\phi(i)] -2[\phi(k)-\phi(i)])e_i.
\end{align*}
Choosing $Z=e_{ik/kj}-e_{ik/ki}$ we find $a(Z)\phi = L^{ij}\phi,$ hence $a$ is surjective.
\end{proof}

\begin{remark}
\textit{
A moment's thought reveals that $\Ker(a)=0$ if and only if $G$ consists of isolated edges. Put differently, $\Ker(a)$ is nonempty if and only if it contains a pair of incident edges.
}
\end{remark}

Our strategy is to interpret the first three ways of describing $DOp^2(G)$ laid out in Proposition 3.3 in terms of the fourth, meaning in terms of $\mathcal{X}(T^2G)$ and the basic map $a.$ 

\mpar
Let's start with $DOp^2(G)\cong\mathcal{X}(G^2).$ We examine the generic structure of balls of radius two in $G$ and their decomposition in terms of certain subgraphs called dipoles. Using dipoles, we define a map $c\colon DOp^2(G)\to \mathcal{X}(T^2G)$ with the property that $a\circ c$ is the identity on $DOp^2(G).$ We think of this as a canonical form because it expresses second order differential operators geometrically in terms of second order vector fields. 

\mpar
It's natural to think of balls of radius two as a union of balls of radius one. But a different decomposition involving dipoles is more useful for our purposes. Before introducing dipoles, let's formally define open and closed balls and spheres. While this is elementary, it's helpful to be precise. 

\begin{definition}
1. The path metric on $G$ where the distance between $i,j\in V_G$ is the length of the shortest walk beginning at $i$ and ending at $j$. More formally,
$$ 
d(i,i)=0,\,\,d(i,j) = \text{min}\,\{k\geq 1\mid A^k(i,j)>0\},\, i\neq j,
$$
where $A$ is the adjacency matrix of $G.$
 
\mpar
2. The \textit{closed ball} centered at $i\in V_G$ of radius $n\geq 0$ is the graph $\overline{B}(i,n)$ with vertex set,
 $$
 V_{\overline{B}(i,n)} = \{j\in V_G\mid d(i,j)\leq n\}
 $$
 and edge set.
 $$
 E_{\overline{B}(i,n)} = \{\,\{j,k\}\in E_G\mid j,k\in V_{\overline{B}(i,n)}\}.
 $$
 
 \mpar
3. The \textit{open ball} centered at $i\in V_G$ of radius $n\geq 1$ is the graph $B(i,n)$ with vertex set,
 $$
 V_{B(i,n)} = V_{\overline{B}(i,n)}
 $$
 and edge set,
 $$
 E_{B(i,n)} = E_{\overline{B}(i,n)}\setminus\{\,\{j,k\}\in E_G\mid d(i,j) = d(i,k) = n\}.
 $$
 
\mpar
4. The \textit{sphere} centered at $i\in V_G$ of radius $n\geq 1$ is the graph $S(i,n)$ with vertex set,
 $$
 V_{S(i,n)} = \{ j\in V_G\mid d(i,j) = n\}
 $$
 and edge set,
 $$
 E_{S(i,n)} = \{\,\{j,k\}\in E_G \mid j,k\in V_{S(i,n)}\}.
  $$
 
 \mpar
 5. The \textit{relative complement} of the closed ball of radius 1 is the graph $\overline{B}^c(i,1)$ of $B(i,2)$ whose vertex set is 
 $$
 V_{\overline{B}^c(i,1)} = V_{B(i,2)}\setminus \{i\},
 $$ 
and whose edge set is 
$$
E_{\overline{B}^c(i,1)}=E_{B(i,2)}\setminus E_{\overline{B}(i,1)}.
$$
\end{definition}

\begin{remark}
\textit{1. Note that $\overline{B}(i,n)$ and $S(i,n)$ are the induced subgraphs of their vertices, $E_{B(i,n)}\cap E_{S(i,n)}=\emptyset$, and $\overline{B}(i,n) = B(i,n)\cup S(i,n).$}

\mpar
\textit{2.  The relative complement of the closed ball of radius $1$ is a forest of trees.}

\mpar
\textit{3. Unit spheres are arbitrarily complex since every finite graph is a unit sphere in some finite simple graph. To see it, let $H$ be a graph and consider the graph $JH,$ where,
$V_{JH} = V_H\cup\{ *\}$ and $E_{JH} = E_H\cup\{\,\{*,i\}\mid i\in V_H\}.$ By construction,  $JH = \overline{B}(*,1)$ and $H=S(*,1).$ $JH$ is called the \textit{join} of $H$ over $*.$}
\end{remark}
To help fix ideas, Figure 3 gives a complete picture of a small but generic ball of radius two.
\begin{figure} [h]
\begin{tikzpicture}
\draw[fill=black] (0,0) circle (2pt);
\draw[fill=black] (1,0) circle (2pt);
\draw[fill=black] (1.5,0) circle (2pt);
\draw[fill=black] (2.5,0) circle (2pt);
\draw[fill=black] (0.6,1) circle (2pt);
\draw[thick] (0,0) -- (1,0) -- (0.6,1) -- (0,0);
\draw[thick] (1.5,0) -- (2.5,0);

\draw[fill=black] (0.6,-1) circle (3pt);
\draw[thick] (0.6, -1) -- (0,0);
\draw[thick] (0.6, -1) -- (1,0);
\draw[thick] (0.6, -1) -- (1.5,0);
\draw[thick] (0.6, -1) -- (2.5,0);
\draw[thick] (0.6, -1) -- (0.6,1);

\draw[fill=black] (0,2) circle (2pt);
\draw[fill=black] (1.25, 2) circle (2pt);
\draw[fill=black] (2.5,2) circle (2pt);

\draw[thick] (1.25, 2) -- (1.5, 0);
\draw[thick] (1.25, 2) -- (1, 0);
\draw[thick] (1.25, 2) -- (0.6, 1);

\draw[thick] (0, 2) -- (0,0);
\draw[thick] (0, 2) -- (2.5,0);
\draw[thick] (2.5, 2) -- (1.5,0);
\draw[thick] (2.5, 2) -- (2.5,0);

\draw[fill=black] (3.6,-1) circle (3pt);
\draw[fill=black] (3,0) circle (2pt);
\draw[fill=black] (4,0) circle (2pt);
\draw[fill=black] (4.5,0) circle (2pt);
\draw[fill=black] (5.5,0) circle (2pt);
\draw[fill=black] (3.6,1) circle (2pt);
\draw[thin] (3.6, -1) -- (3,0);
\draw[thin] (3.6, -1) -- (4,0);
\draw[thin] (3.6, -1) -- (4.5,0);
\draw[thin] (3.6, -1) -- (4.5,0);
\draw[thin] (3.6, -1) -- (3.6,1);
\draw[thin] (3.6, -1) -- (5.5,0);

\draw[fill=black] (6,0) circle (2pt);
\draw[fill=black] (7,0) circle (2pt);
\draw[fill=black] (6.6,1) circle (2pt);
\draw[fill=black] (7.5,0) circle (2pt);
\draw[fill=black] (8.5,0) circle (2pt);

\draw[thick] (6,0) -- (7,0) -- (6.6,1) -- (6,0);
\draw[thick] (7.5,0) -- (8.5,0);

\draw[fill=black] (9,0) circle (2pt);
\draw[fill=black] (10,0) circle (2pt);
\draw[fill=black] (10.5,0) circle (2pt);
\draw[fill=black] (11.5,0) circle (2pt);
\draw[fill=black] (9.6,1) circle (2pt);

\draw[fill=black] (9,2) circle (2pt);
\draw[fill=black] (10.25, 2) circle (2pt);
\draw[fill=black] (11.5,2) circle (2pt);
\draw[thin] (10.25, 2) -- (10.5, 0);
\draw[thin] (10.25, 2) -- (10, 0);
\draw[thin] (10.25, 2) -- (9.6, 1);

\draw[thin] (9, 2) -- (9,0);
\draw[thin] (9, 2) -- (11.5,0);
\draw[thin] (11.5, 2) -- (10.5,0);
\draw[thin] (11.5, 2) -- (11.5,0);

\draw[fill=black] (0,-3.5) circle (2pt);
\draw[fill=black] (1,-3.5) circle (2pt);
\draw[fill=black] (0.6,-2.5) circle (2pt);
\draw[fill=black] (0,-1.5) circle (2pt);
\draw[fill=black] (0.6,-4.5) circle (3pt);
\draw[thick] (0,-3.5) -- (0, -1.5);
\draw[thick] (0,-3.5) -- (0.6,-4.5);
\draw[thick] (0,-3.5) -- (0.6,-2.5);
\draw[thick] (0,-3.5) -- (1,-3.5);

\draw[fill=black] (2,-3.5) circle (2pt);
\draw[fill=black] (3,-3.5) circle (2pt);
\draw[fill=black] (2.6,-2.5) circle (2pt);
\draw[fill=black] (3.25, -1.5) circle (2pt);
\draw[fill=black] (2.6,-4.5) circle (3pt);
\draw[thick] (2.6, -2.5) -- (2, -3.5);
\draw[thick] (2.6, -2.5) -- (2.6, -4.5);
\draw[thick] (2.6, -2.5) -- (3, -3.5);
\draw[thick] (2.6, -2.5) -- (3.25, -1.5);

\draw[fill=black] (4,-3.5) circle (2pt);
\draw[fill=black] (5,-3.5) circle (2pt);
\draw[fill=black] (4.6,-2.5) circle (2pt);
\draw[fill=black] (5.3, -1.5) circle (2pt);
\draw[fill=black] (4.6,-4.5) circle (3pt);
\draw[thick] (5, -3.5) -- (4.6, -4.5);
\draw[thick] (5, -3.5) -- (4, -3.5);
\draw[thick] (5, -3.5) -- (4.6, -2.5);
\draw[thick] (5, -3.5) -- (5.3, -1.5);

\draw[fill=black] (7.25,-3.5) circle (2pt);
\draw[fill=black] (8.25,-3.5) circle (2pt);
\draw[fill=black] (7.05, -1.5) circle (2pt);
\draw[fill=black] (8.25,-1.5) circle (2pt);
\draw[fill=black] (6.05,-4.5) circle (3pt);
\draw[thick] (7.2, -3.5) -- (6.05, -4.5);
\draw[thick] (7.25, -3.5) -- (8.255, -3.5);
\draw[thick] (7.25, -3.5) -- (8.25, -1.5);
\draw[thick] (7.25, -3.5) -- (7.05, -1.5);
\draw[fill=black] (10.5,-3.5) circle (2pt);
\draw[fill=black] (9,-1.5) circle (2pt);
\draw[fill=black] (11.5,-1.5) circle (2pt);
\draw[fill=black] (9.55,-4.5) circle (3pt);
\draw[thick] (11.5, -3.5) -- (9.55, -4.5);
\draw[thick] (11.5, -3.5) -- (9, -1.5);
\draw[thick] (11.5, -3.5) -- (11.5, -1.5);
\draw[thick] (11.5, -3.5) -- (10.5, -3.5);

\end{tikzpicture}
\caption{\textit{Top row (l-r)}: A ball of radius two centered at the thick vertex, the ball of radius 1 centered at the thick vertex, the sphere of radius 1, and the relative complement of the closed ball of radius 1. Note that the unit sphere is not connected as it is the disjoint union of a triangle and an edge The ball has 9 vertices, 16 edges, 7 triangles, 5 rectangles and has cyclomatic number of 9.  \textit{Bottom row}: The balls of radius 1 centered at nearest neighbors of the thick vertex. }
\end{figure}
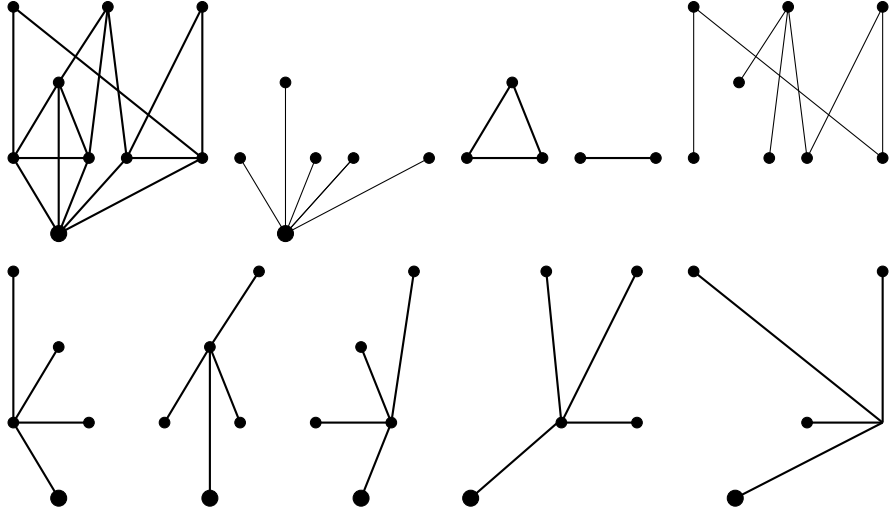

\begin{remark}
\textit{It may be helpful to glance at Figure 3 when considering the following general observations. It's elementary that $B(i,1)\cong K_{1,d},$ where $d=\deg(i),$ and,
$$
B(i,2)=\bigcup_{\pi(u)=i}B(\pi_+(u),1).
$$
Hence every ball of radius two is the union of a set of star graphs. Evidently, $B(i,1)$ plays a privileged role since it shares exactly one edge with each of the other stars and this is why we don't need to include it in the union above. 
If there are no other relations among these unit balls, meaning the stars centered at nearest neighbors of $i$ have no other vertices or edges in common, then $B(i,2)$ is a tree. If, say, $B(j,1)$ and $B(k,1)$  have an edge in common then $\{i, j, k\}$ is a triangle, and $\{j,k\}$ is an edge of $S(i,1).$ If they have a vertex in common, say $l$, then $\{i,j,k,l\}$ are the vertices of a rectangle. Note that if a triangle based at $i$ has two edges in common with a rectangle based at $i$ then there is a second triangle in $B(i,2)$ complementary to the first one. The two triangles have an edge in common which lies in the unit sphere. Note that $B(i,2)$ can't contain an induced $k$-cycle for $k\geq 5$ since it would then have a vertex at a distance greater than two from $i$. Thus, the topology of $B(i,2)$ is completely determined by the intersections between the unit balls of the nearest neighbors of vertex $i.$}
\end{remark} 

\mpar
We can think of a generic ball $B$ of radius two in a complementary way by focusing on its unit sphere which is a finite simple graph, say $S$. Let vertex $i$ be the center of $B,$ let $j_1, j_2,\cdots, j_n$ be the vertices of $B$ at distance two from $i,$ and let $V_l\subset V_S$ be the nearest neighbors of $j_l$ in $B.$ Finally, let $J_0$ be the join of $S$ over $i$ and let $J_l$ be the join of $V_l$ over $j_l,$ for $1\leq l\leq n.$ Then, 
$$
B=H\cup J_0\cup J_1\cdots\cup J_n
$$

This way of looking at $B(i,2)$ motivates the definition of a dipole. If $d(i,j)=2$ then there is a set $\{k_1, k_2,\cdots  k_m\}\subset V_{S(i,1)}$ such that $\{i, k_l\}, \{k_l ,j\}\in E_G,\, 1\leq l\leq m.$ The union of these edges is the dipole of $i$ and $k,$ denoted $Dp(i,j).$ (See Figure 4).

\begin{definition}
1. The \textit{dipole} of order $n\geq 1$ is the unique graph $Dp_n$ having $n+2$ vertices, $2n$ edges, two vertices of degree $n,$ and $n$ vertices of degree $2$. The vertices of degree $n$ are called \textit{poles}.

\mpar
2. If $Dp$ is a dipole subgraph of $G$ not contained in a dipole having the same poles but more edges, then $Dp$ is said to be a \textit{proper dipole} of $G.$ 
\end{definition}

We  state the following dipole decomposition without proof as it follows directly from the definitions. 

\begin{proposition}
Let $B(i,2)$ be a ball of radius two in some graph $G$. Let $j_1, \cdots, j_n$ be an enumeration of all vertices $j\in V_{B(i,2)}$ such that $d(i,j)=2$ and let $Dp(i, j_1), \cdots, Dp(i, j_n)$ be the the associated proper dipoles. Then $S(i,1)$ is edge-disjoint from the  dipoles and,
$$
B(i,2) = S(i,1)\cup Dp(i, j_1)\cup\cdots\cup Dp(i,j_n).
$$
\end{proposition}

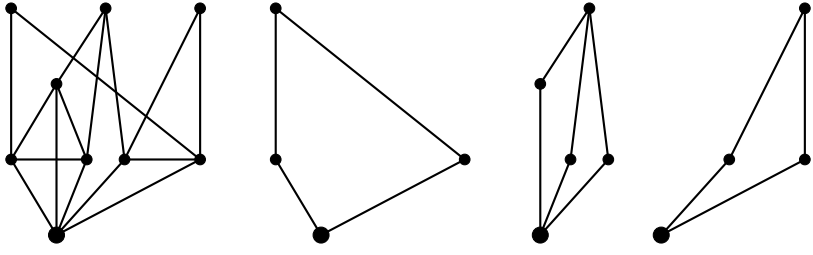
\begin{figure} 
\begin{tikzpicture}
\draw[fill=black] (0,0) circle (2pt);
\draw[fill=black] (1,0) circle (2pt);
\draw[fill=black] (1.5,0) circle (2pt);
\draw[fill=black] (2.5,0) circle (2pt);
\draw[fill=black] (0.6,1) circle (2pt);
\draw[thick] (0,0) -- (1,0) -- (0.6,1) -- (0,0);
\draw[thick] (1.5,0) -- (2.5,0);

\draw[fill=black] (0.6,-1) circle (3pt);
\draw[thick] (0.6, -1) -- (0,0);
\draw[thick] (0.6, -1) -- (1,0);
\draw[thick] (0.6, -1) -- (1.5,0);
\draw[thick] (0.6, -1) -- (2.5,0);
\draw[thick] (0.6, -1) -- (0.6,1);

\draw[fill=black] (0,2) circle (2pt);
\draw[fill=black] (1.25, 2) circle (2pt);
\draw[fill=black] (2.5,2) circle (2pt);
\draw[thick] (1.25, 2) -- (1.5, 0);
\draw[thick] (1.25, 2) -- (1, 0);
\draw[thick] (1.25, 2) -- (0.6, 1);
\draw[thick] (0, 2) -- (0,0);
\draw[thick] (0, 2) -- (2.5,0);
\draw[thick] (2.5, 2) -- (1.5,0);
\draw[thick] (2.5, 2) -- (2.5,0);

\draw[fill=black] (3.5,0) circle (2pt);
\draw[fill=black] (6,0) circle (2pt);
\draw[fill=black] (3.5,2) circle (2pt);
\draw[fill=black] (4.1,-1) circle (3pt);
\draw[thick] (4.1, -1) -- (3.5,0) -- (3.5,2);
\draw[thick] (4.1, -1) -- (6,0) -- (3.5,2);

\draw[fill=black] (7.4,0) circle (2pt);
\draw[fill=black] (7.9,0) circle (2pt);
\draw[fill=black] (7,1) circle (2pt);
\draw[fill=black] (7.65, 2) circle (2pt);
\draw[fill=black] (7,-1) circle (3pt);
\draw[thick] (7, -1) -- (7,1) -- (7.65, 2);
\draw[thick] (7, -1) -- (7.4,0) -- (7.65, 2);
\draw[thick] (7, -1) -- (7.9,0) -- (7.65, 2);

\draw[fill=black] (9.5,0) circle (2pt);
\draw[fill=black] (10.5,0) circle (2pt);
\draw[fill=black] (10.5,2) circle (2pt);
\draw[fill=black] (8.6,-1) circle (3pt);
\draw[thick] (8.6, -1) -- (9.5,0) -- (10.5, 2) --(10.5,0) -- (8.6, -1);
\end{tikzpicture}
\caption{A ball of radius 2 centered at the thick vertex and its proper dipoles.}
\end{figure}

\begin{remark}
\textit{1. Note that two distinct dipoles in $G$ may have edges and vertices in common, indeed one may be a subgraph of the other. In particular, if two dipoles have vertices in common, a pole of one of them need not be a pole of the other.}

\mpar
\textit{2. There is a one-to-one correspondence between the set of proper dipoles in $G$ and $E_{G^2}\negthickspace\setminus\negthickspace E_G,$ the set of edges in $G^2$ not in $G.$ This shows the edges of $G^2$ can be recovered from the edges and dipoles of $G.$}
\end{remark}

The dipole decomposition described above leads to a canonical form for $DOp^2(G).$

\begin{proposition}
(First Canonical Form). Let $\gamma\colon DOp^2(G)\to\mathcal{X}(G^2)$ be the canonical isomorphism,
$$
\gamma(L)=\sum_{u\in V(tG^2)}L(\pi(u), \pi_+(u))e_u
$$
and define a map $\delta\colon\mathcal{X}(G^2)\to\mathcal{X}(T^2G)$ by the following rule. For every $i,j\in V_G$ such that $d(i,j)=1,$ let, 
$$
\delta(e_{ij})=-\tfrac{1}{2}e_{ij/ji}.
$$
For every $i, j\in V_G$ such that $d(i,j)=2,$ let,
$$
\delta(e_{ij})=\frac{1}{A^2(i,j)}\sum_{k\in V_G} A(i,k)A(k,j)[e_{ik/kj}-e_{ik/ki}],
$$
where $A$ is the adjacency matrix of $G.$ Then $c=\delta\circ\gamma\colon DOp^2(G)\to\mathcal{X}(T^2G)$ and $a\circ c$ is the identity.
\end{proposition}

\begin{proof}
Observe that if $d(i,j)=1$ then $a(-\tfrac{1}{2}e_{ik/ki})\phi = (\phi(j)-\phi(i))e_i=L^{ij}\phi$ and if $d(i,j)=2$ then $a(e_{ik/kj}-e_{ik/ki})\phi = (\phi(j)-\phi(i)e_i=L^{ij}$, as well. It follows that $a(\delta(e_{u}))\phi = e_u$ for every $u\in V_{tG^2},$ and therefore $a(\delta\circ\gamma(L))=L$ for all $L\in DOp^2(G).$
\end{proof}

\section{Cardinality of $DOp^2(G)$}

There are elementary formulas for the cardinality of the space of second order differential operators in terms of the adjacency matrix.

\begin{proposition}
Let $G$ be a finite simple graph. Then,
\begin{align*}
|DOp^2(G)| & =\sum_{i\in V_G}\sum_{j\in V_G} A(i,j) + sgn(A^2(i,j))(1-A(i,j))\\
&  =\sum_{i\in V_G} (|V_{B(i,2)}| -1),
\end{align*}
where $sgn(x)=1$, if $x>0, sgn(x) =-1$ if $x<0,$ and $sgn(0)=0.$
\end{proposition}

\begin{proof}
We know $|DOp^2(G)| = |\mathcal{X}(G^2)|=2|E_{G^2}|$ and we also know $\{i,j\}\in E_{G^2}$ if and only if $d(i,j)=1,2.$ Evidently, $A(i,j)=1$ if and only if $d(i,j)=1$ and $sgn(A^2(i,j))(1-A(i,j))=1$ if and only if $d(i,j)=2,$ and they equal zero, otherwise. Therefore, the first sum above is counting the edges of $E_{G^2}$ but with multiplicity $2$ because of the symmetry of the adjacency matrix. On the other hand, recall that the set of operators $L^{ij}, d(i,j)=1,2,$ is a basis hence of $DOp^2(G)$. For each fixed $i\in G$ there are $|V_{B(i,2)}-1|$ such vertices and the second formula follows from this.
\end{proof}

So the size of $DOp^2(G)$ is easy to understand analytically. But we would also like to understand it geometrically in terms of the second tangent bundle. Observe that since $a$ is surjective, $|DOp^2(G)|+|\Ker(a)|= |\mathcal{X}(T^2G)|.$ It's possible to evaluate the right hand side in geometric terms. Recall,
$$
\mathcal{X}(T^2G)\cong\bigoplus_{i\in V_G}T^2_i(G)\cong\bigoplus_{i\in V_G}\bigoplus_{\pi(u)=i}T_u(tG) \cong \negthickspace\bigoplus_{u\in V_{tG}} T_u(tG) \cong\mathcal{X}(tG).
$$
Thus  $|DOp^2(G)|+|\Ker(a)|=2|E_{tG}|.$ Let's calcuate this another way.

\begin{theorem}
1. For any connected, finite simple graph $G,$ let $T_G$ be the set of all  triangle subgraphs, and let $\xi(H)=|E_H|-|V_H|+1$ be the cyclomatic number of a graph. Then,
$$
|E_{tG}|+|E_G| = |DOp^2(G)| +3|T_G| +\sum_{i\in V_G}\xi(B(i,2)).
$$
Therefore,  
$$
2|E_G|\leq |DOp^2(G)|\leq |E_{tG}|+|E_G|
$$
with equality on the left if and only if $G$ is a complete graph and equality on the right if and only if for every $i\in V_G, B(i,2)$ is a tree or, equivalently, $G$ has no cycles of length $3$ or $4.$

\mpar
2. The dimension of the kernel of $a$ is,
$$
|\Ker(a)| = |E_{tG}|-|E_G| +3|T_G| +\sum_{i\in V_G}\xi(B(i,2))
$$
and satisfies the inequality,
$$
|E_{tG}|-|E_G|\leq |\Ker(a)|\leq 2\left(|E_{tG}|-|E_G|\right),
$$
with equality on the left if and only if for every $i\in V_G, B(i,2)$ is a tree and equality on the right if and only if $G$ is a complete graph.
\end{theorem}

\begin{proof}
First, observe that the adjacency matrix acts on functions. When applied to the degree function we have,
$$
A\,\Deg(i)=\negthickspace\sum_{\pi(u)=i}\Deg({\pi_+(u)}) = |E_{B(i,2)}| + |E_{S(i,1)}|,
$$
because the sum counts every edge in $B(i,2)$ once with the exception of edges in $S(i,1)$ which are counted twice.  Let's calculate $\sum_{i\in V_G} |E_{B(i,2)}| + |E_{S(i,1)}|$ in two ways. First,
\begin{align*}
\sum_{i\in V_G} |E_{B(i,2)}| + |E_{S(i,1)}| & =\sum_{i\in{V_G}}A \,\Deg(i) = \sum_{i\in V_G}\sum_{j\in V_G}A(i,j)\Deg(j)\\ 
& = \sum_{j\in V_G}\Deg^2(j) = |E_{tG}| + |E_G|.
\end{align*}
(See Appendix A(i) for the last step). Now, every edge in $S(i,1)$ determines a unique triangle rooted at $i$. If we denote the set of all such rooted triangles by $T_G(i)$ then  $|E_{S(i,1)}| = |T_G(i)|.$ Also, let's recall that $|DOp^2(G)|=\sum_{i\in V_G} (|V_{B(i, 2)}|-1)|$, by Proposition 3.1 Thus, using the definition of cyclomatic number we find,
\begin{align*}
|E_{tG}|+|E_G| & = \sum_{i\in V_G}|E_{B(i,2)}| + |E_{S(i,1)}|\\
& = \sum_{i\in V_G}\left(\xi(B(i,2))+ |V_{B(i,2)}| - 1\right) + \sum_{i\in V_G} |T_G(i)|\\ 
& = |DOp^2(G)| + 3|T_G| + \sum_{i\in V_G}\xi(B(i,2)),
\end{align*}
noting that every triangle in $G$ is counted exactly three times in the sum above.

\mpar
The upper bound on $|DOp^2(G)|$ follows directly from this identity and the lower bound comes from the fact that $2|E_G| = |DOp^1(G)|\leq |DOp^2(G)|.$  The cases of equality are also straighforward. First, $DOp^1(G)=DOp^2(G)$ if and only if there is no pair of vertices such that $d(i,j)=2$, meaning every vertex is adjacent to every other vertex. Second, every ball of radius 2 is a tree if and only if there are no $C_3$ or $C_4$ subgraphs or, equivalently, $\xi(B(i,2))=0$ for every vertex $i.$ Since $|\Ker(a)|=2|E_{tG}| - |DOp^2(G)|, $ item 2 is a direct consequence of item 1.
\end{proof}

\begin{remark}
\textit{1. The condition that $B(i,2)$ is a tree for every vertex $i$ is equivalent to the condition that $G$ contains no cycles of length 3 or 4.}

\mpar
\textit{ 2. It's well known how to calculate $3|T_G|$ in terms of the adjacency matrix. Observe that,
$$
A^3(i,i)=\sum_{j\in V_G}\sum_{k\in V_G}A(i,j)A(j,k)A(k,i),
$$
and that a term in the sum equals $1$ if $\{i,j,k\}$ is a triangle and $0$ otherwise. By symmetry of $A(j,k),$ the sum equals $2|T_G(i)|.$ Observing that a fixed triangle contributes to the local triangle count at each of its vertices we find, 
$$
3|T_G|= \sum_{i\in V_G} |T_G(i)|= \tfrac{1}{2}tr A^3.
$$
Therefore, all the intricacy of $|DOp^2(G)|$ is captured in the sum $\sum_{i\in V_G}\xi(B(i,2)).$ }

\mpar
\textit{3. The formula for $|\ker(a)|$ consists of a positive term $|E_{tG}|-|E_G|,$ equal to the number of transverse  edges of $tG,$ plus two nonnegative terms: $3|T(G)|,$ thrice the number of triangles, and $\sum_{i\in V_G}\xi(B(i,2)),$ the sum of the number of linearly independent cycles in each ball of radius two. Thus the formula, in principle, labels a basis for $\Ker(a)$ in terms of geometric objects. However, since balls of radius 2 overlap in complicated ways it's not clear there is such a labelling in practice.}
\end{remark}
 
\section{Second Canonical Form via the Helmholtz Decomposition}
We've seen that $a$ always has a non-trivial kernel, unless $G$ consists of a single edge. Since understanding $\Ker(a)$ is equivalent to understanding $DOp^2(G)$ in geometric terms let's turn our attention there. The following definition introduces the terms of argument.

\begin{definition}
1. For each $Z\in\mathcal{X}(T^2G)$ and each $i\in V_G,$  define the section $Z_i=(e_i\circ\pi^2)\cdot Z\in\mathcal{X}(T^2G)$ by the rule,
$$
Z_i(\alpha)=e_i(\pi^2(\alpha))Z(\alpha)=
\begin{cases}
Z(\alpha), & \pi^2(\alpha)=1.\\
0, &\text{otherwise}.
\end{cases}
$$
\mpar
2. The \textit{generalized divergence} $\DIV\colon\mathcal{X}(T^2(G))\to\mathcal{X}(Ch^2(G))$ is the operator,
$$
\left(\DIV Z\right)_i(j)=\Div\circ\Div_{tG}\left(Z_i\right)(j).
$$

\mpar
3.  The \textit{Hessian}   $\Hess\colon\mathcal{X}(Ch^2(G))\to\mathcal{X}(T^2(G))$ is the operator,
$$
\Hess\zeta(\alpha)=d^2\zeta_{\pi^2(\alpha)}(\alpha).
$$

\mpar
4. The \textit{flux}  $\lozenge\colon\mathcal{X}(T^2(G))\to\mathcal{X}(Lf^2(G))$ is the operator,
$$
\left(\lozenge Z\right)_i(j)=\sum_{\substack{\pi^2(\alpha)=i\\
\pi_+^2(\alpha)=j}} Z(\alpha).
$$
\end{definition}

\begin{remark}
\textit{1. To reduce notational clutter we often write $\left(\DIV Z\right)_i(j)=\DIV Z_i(j)$ and $\left(\lozenge Z\right)_i(j)=\lozenge Z_i(j),$ although it is a slight abuse of notation.}

\mpar
\textit{2. In \cite{M1} we defined $\Hess$ as a map from $C(G)$ to $\mathcal{X}(T^2(G))$ and now we are using the same notation for a map from $\mathcal{X}(Ch^2G)$ to $\mathcal{X}(T^2(G)).$ While this usage is inconsistent, the two maps are closely related and no confusion is likely to arise.}

\mpar
\textit{3. Note that  $\lozenge Z_i(j),\, i\neq j,$ is the sum of terms of the form $Z(ik/kj),$ where $\{i,k\}$ and $\{k,j\}$ are edges of $G,$ and $\lozenge Z_i(i)$ is the sum of terms of the form $Z(ij/ki)$, where $\{i,j\}$ and $\{i,k\}$  are edges of $G.$ If we think of $Z(\alpha)$ as a rate or a flux, then $\lozenge Z_i(j)$ is the total flux from $i$ to $j.$ When $d(i,j)=2,$ the sum defining $\lozenge Z_i(j)$ is indexed by incident edges in the dipole $Dp(i,j).$ Observe that $\lozenge Z_i(j)\neq\lozenge Z_j(i),$ in general.}
\end{remark}

The following proposition characterizes $\ker(a)$, provides a formula for calculating $\DIV Z,$ and shows that $\Hess$ and $\DIV$ are mutually adjoint. 

\mpar
In preparation, it's helpful to review some ideas and notation. Recall that the tangent graph is functorial in the sense that if $h\colon G\to K$ is a graph homomorphism then $dh\colon tG \to tK$ is also a graph homomorphism where $dh(ij)=h(i)h(j).$ (See Appendix A(i)). When applied to the involution $\sigma\colon tG\to tG$ this observation yields the involution,
$$
d\sigma\colon t^2G\to t^2G
$$
where $d\sigma(ij/kl) =\sigma(ij)\sigma(kl) = ji/lk.$ (Recall that $ij/kl\in V_{t^2G}$ if and only if $\{i,j\},\,\{k,l\}\in E_G$ and either $j=k$ or $i=l$). Since $t^2G$ is the tangent graph of $tG$ it has its own involution $\sigma_{tG}\colon t^2G\to t^2G$ where $\sigma_{tG} (ij/kl) = kl/ij.$  Obviously, $\sigma_{tG}\neq\ d\sigma$ and, slightly less obviously, they generate a group of symmetries of $t^2G$ isomorphic to $\mathbb{Z}_2\times\mathbb{Z}_2.$ In fact,
$$
\sigma_{tG}\circ d\sigma(ij/lk)=lk/ji=d\sigma\circ\sigma_{tG}(ij/kl)
$$
which shows these involutions commute. By abuse of notation we often drop the subscript and write the involutions as $\sigma$ and $d\sigma.$ While this is ambiguous, context usually makes usage clear.

\mpar
It's convenient to extend the graph homomorphism $\pi\colon tG\to G$ to a map $\pi\colon\mathcal{X}(G)\to C(G) $ by the rule $\pi X(i)=\sum_{\pi(u)=i}X(u).$ When applied to the morphism $\pi_+\colon tG\to G$ we have the extension $\pi_+X(i)= \sum_{\pi_+(u)=i}X(u)$ and in this notation we can write the divergence operator as,
$$
\Div X(i)=\pi_+X(i)-\pi X(i).
$$
By the same token we can extend the graph homomorphism
$\pi_{tG}\colon t^2G\to tG$ to a map  $\pi_{tG}\colon\mathcal{X}(tG)\to\mathcal{X}(G)$ where,
$$
\pi_{tG}Z(u)=\sum_{\pi_{tG}(\alpha)=u}Z(\alpha).
$$
There is also the graph homomorphism $d\pi\colon t^2G\to tG$ given by the rule,
$$
d\pi(ij/kl)=\pi(ij)\pi(kl)=ik
$$
and its extension $d\pi\colon\mathcal{X}(tG)\to\mathcal{X}(G)$ where $d\pi Z(u)=\sum_{d\pi(\alpha)=u} Z(\alpha).$ Note that $\mathcal{X}(tG)\cong\mathcal{X}(T^2G)$ so these operators are well defined on second order vector fields. Similar considerations apply to $d\pi_+X(u)=\sum_{d\pi_+(\alpha)=u} Z(\alpha).$

\mpar
Finally, compositions of  $\pi_{tG}, \pi_{tG\, +}$ followed by $\pi, \pi_+$ occur frequently. To avoid notational clutter we often drop the subscript $tG$  and write $\pi^2=\pi\circ\pi_{tG}$ and $\pi_+^2=\pi\circ\pi_{tG\, +},$ etc., when the context is clear. For example, we have,
\begin{align*}
\Div_{tG}\Div & = (\pi_+-\pi)(\pi_+-\pi)\\
&  =\pi_+^2-\pi\circ\pi_+-\pi_+\circ\pi +\pi^2.
\end{align*}

\begin{proposition}
1. We have,
$$
a(Z)\phi(i)=\sum_{j\in V_G}\DIV Z_i(j)\phi(j),
$$
hence, $\Ker(a)=\Ker(\DIV).$

\mpar
2. The generalized divergence $\DIV$ can be calculated in terms of the flux $\lozenge$ in the sense that,
$$
\DIV Z_i(j)= \lozenge Z_i(j) +
\begin{cases}
\pi^2 Z(i), & \text{$j=i$},\\
 -\pi Z(ij) - d\pi Z(ij), & \text{$d(i,j)=1$},\\
0, & \text{$d(i,j)=2$.}
\end{cases}
$$

\mpar
3. $a(Z)\in DOp^1(G)$ if and only if $\DIV Z\in \mathcal{X}(Ch^1G)$ or, equivalently, $\lozenge Z_i(j)=0$ for all $d(i,j)=2.$ Specifically, $a(Z)$ is a vector field if and only if for all $d(i,j)=2,$
$$
\sum_{k\in V_G}A(i,k)A(k,j)Z(ik/kj)=0.
$$
\end{proposition}

\begin{proof}
Item 1 is a straightforward calculation using the fact that $\nabla_{tG}\nabla$ and $\Div\Div_{tG}$ are adjoints of one another, namely,
\begin{align*}
a(Z)\phi(i) & = \sum_{\pi^2(\alpha)=i}Z(\alpha)d^2\phi(\alpha)\\
& = \sum_{\alpha\in V_{t^2G}}Z_i(\alpha)\nabla_{tG}\nabla\phi(\alpha)\\
& = \sum_{j\in V_G}\Div\Div_{tG}Z_i(u)\phi(j) = \sum_{j\in V_G}\DIV Z_i(j)\phi(j).
\end{align*}
To calculate $\DIV Z_i(j)$ observe that for any $W\in\mathcal{X}(T^2G)$ we have,
\begin{align*}
\Div\Div_{tG}  W & = \sum_{\alpha\in V_{t^2G}}W(\alpha)\Div\Div_{tG}e_{\alpha}\\
& = \sum_{\alpha\in V_{t^2G}}W(\alpha)[e_{\pi_+^2(\alpha)}-e_{\pi\circ\pi_+(\alpha)} - e_{\pi_+\circ\pi(\alpha)} + e_{\pi^2(\alpha)}],
\end{align*}
hence,
$$
\Div\Div_{tG}W(j)=\sum_{\pi_+^2(\alpha)=j} \negthickspace\negthickspace W(\alpha)\thickspace - \negthickspace\sum_{\pi\circ\pi_+(\beta)=j} \negthickspace\negthickspace W(\beta)\thickspace - \negthickspace\sum_{\pi_+\circ\pi(\gamma)=j} \negthickspace\negthickspace W(\gamma)\thickspace + \negthickspace\sum_{\pi^2(\delta)=j} \negthickspace W(\delta).
$$
On substituting $W=Z_i=(e_i\circ\pi^2)\cdot Z$ we have,
\begin{align*}
(\DIV Z)_i(j) & = \Div\Div_{tG}( Z_i)(j)\\
& = \sum_{\substack{\pi^2(\alpha)=i\\\pi_+^2(\alpha)=j}} \negthickspace\negthickspace Z(\alpha)\thickspace - \negthickspace\sum_{\substack{\pi^2(\beta)=i\\ \pi\circ\pi_+(\beta)=j}} \negthickspace\negthickspace Z(\beta)\thickspace - \negthickspace\sum_{\substack{\pi^2(\gamma)=i\\\pi_+\circ\pi(\gamma)=j}} \negthickspace\negthickspace Z(\gamma)\thickspace + \negthickspace\sum_{\substack{\pi^2(\delta)=i\\\pi^2(\delta)=j}} \negthickspace Z(\delta),\\
& = I - II - III + IV.
\end{align*}
The next step is to evaluate these sums when $d(i,j)=0, 1, 2.$ First, observe that $I=\lozenge Z_i(j)$ for all possible values of $i$ and $j$ and that $IV=\pi^2 Z(i)$ if $i=j$  and $IV = 0$ if $i\neq j$. 

\mpar
Note that $II$ and $III$ are empty sums when $i=j$ hence $\DIV Z_i(i)=\lozenge Z_i(i) +\pi^2 Z(i).$ 

\mpar
Next, if $d(i,j)=2$ then $II$ and $II$ are also empty sums hence $\DIV Z_i(j)=\lozenge Z_i(j).$ 

\mpar
So, suppose $d(i,j)=1.$ Then,
\begin{align*}
II & = \sum_{d\pi(\beta)=ij} Z(\beta) = d\pi Z(ij)\\
III & = \sum_{\pi(\gamma)=ij} Z(\gamma) = \pi Z(ij)
\end{align*}
and item 2 follows from this. Item 3 follows from the definitions and this completes the proof. 
\end{proof} 

\begin{proposition}
1. $\Hess\colon\mathcal{X}(Ch^2(G))\to\mathcal{X}(T^2(G))$ is injective. 
\mpar
2. $\Hess$ and $\DIV$ are mutually adjoint in the sense that for all $\zeta\in\mathcal{X}(Ch^2(G))$ and $Z\in\mathcal{X}(T^2(G))$ we have,
$$
\langle\Hess\zeta, Z\rangle_{\mathcal{X}(T^2(G))} = \langle\zeta, \DIV Z\rangle_{\mathcal{X}(Ch^2(G))}.
$$
In particular, $\Ker(\DIV)=\Image(\Hess)^{\perp}$ with respect to the canonical inner product on $\mathcal{X}(T^2(G))$ and $\DIV\colon\Image(\Hess)\to\mathcal{X}(Ch^2(G))$ is an isomorphism.

\mpar
3. The operator $\Delta_{Ch^2(G)} =\DIV\circ\Hess$ is the Laplacian acting on sections of the $2$-charge bundle. It is a positive, self-adjoint operator and an isomorphism of $\mathcal{X}(Ch^2(G)).$
\end{proposition}

\begin{proof}
To show $\Hess$ is injective let's suppose $d^2\zeta_{\pi^2(\alpha)}(\alpha)=0$ for all $\alpha\in V_{t^2G}.$  If $\{i,j\}\in E_G$ then,
$$
d^2\zeta_i(ij/ji) = 2(\zeta_i(i)-\zeta_i(j)) = 0,
$$
hence $\zeta_i(j)=\zeta_i(i)$ for all $j\in V_G$ with $d(i,j)=1.$ Next, suppose $d(i,j)=2.$ Then there is a vertex $k$ such that $\{i,j\}, \{k,j\}\in E_G.$ Keeping in mind that $d(i,k)=1$ we have,
$$
d^2\zeta_i(ik/kj) = \zeta_i(j)-2\zeta_i(k) + \zeta_i(i)= \zeta_i(j)-\zeta_i(i) = 0, 
$$
Hence $\zeta_i(j)=\zeta_i(i)$ for all vertices such that $d(i,j)\leq 2.$ Since $\zeta_i $ is constant on its support and has zero average, it vanishes identically and therefore $\Hess$ is injective. 

\mpar
Item 2 is a straightforward consequence of the mutual adjointness of gradient and divergence on $G$ and $tG$ but it's helpful to write it out, nonetheless. Recalling that $d^2\phi(\alpha)= \nabla_{tG} \nabla\phi(\alpha)$ we have,
\begin{align*}
\langle \Hess\zeta, Z\rangle_{\mathcal{X}(T^2G)} & = \sum_{\alpha\in V_{t^2G}}d^2\zeta_{\pi^2(\alpha)}(\alpha) Z(\alpha)\\
& = \sum_{i\in V_G}\sum_{\alpha\in V_{t^2G}}e_i(\pi^2(\alpha))d^2\zeta_i(\alpha) Z(\alpha)\\
& = \sum_{i\in V_G}\sum_{\alpha\in V_{t^2G}}d^2\zeta_i(\alpha) Z_i(\alpha)
\end{align*}

Fix $i\in V_G.$ Letting $\phi=\zeta_i$ and $W=Z_i$ we have,
\begin{align*}
\sum_{\alpha\in V_{t^2G}}d^2\zeta_i(\alpha) Z_i(\alpha) & = \langle \nabla_{tG}\nabla\phi, W\rangle_{\mathcal{X}(tG)}\\
\vspace{-1 em}
&=\langle\phi,\Div\Div_{tG}W\rangle_{C(G)}\\ 
& =\sum_{j\in V_G}\zeta_i(j)\DIV Z_i(j).
\end{align*}
Therefore,
\begin{align*}
\langle \Hess\zeta, Z\rangle_{\mathcal{X}(T^2G)} & = \sum_{i\in V_G}\sum_{j\in V_G} \zeta_i(j)\DIV Z_i(j)\\
& = \sum_{i\in V_G}\langle\zeta_i,\DIV Z_i\rangle_{Ch_i^2(G)}\\
& = \langle \zeta, \DIV Z\rangle_{\mathcal{X}(Ch^2(G))}.
\end{align*}

\mpar
By familiar arguments this identity implies $\Ker(\DIV)=\Image(\Hess)^{\perp}.$ But then we have $\Image(\Hess)=\Ker(\DIV)^{\perp} =\Ker(a)^{\perp},$ by Proposition 5.3.1. Thus, the range of $\DIV$ restricted to $\Image(\Hess)$ is isomorphic to the range of $a$ restricted to $\Ker(a)^{\perp}$ which equals $DOp^2(G)$ Since $DOp^2(G)\cong\mathcal{X}(Ch^2(G))$ this shows $\DIV$ is an isomorphism from $\Image(\Hess)$ to $\mathcal{X}(Ch^2(G)).$

\mpar
That the Laplacian is non-negative and self-adjoint is a consequence of the fact that $\Hess$ and $\DIV$ are mutually adjoint, by standard arguments.  By items 1 and 2, $\Hess$ is injective and $\DIV$ restricted to $\Image(\Hess)$ is an isomorphism hence $\Delta_{Ch^2(G)}$ is a bijection. Finally, zero can't be in the spectrum of $\Delta_{Ch^2(G)}$ because it's injective which means it's a positive operator.
\end{proof}
\begin{theorem}
(Helmholtz Decomposition). The operator $p_H=\Hess\circ\Delta_{Ch^2G}^{-1}\circ\DIV$ is an orthogonal projection. Therefore,
$$
\mathcal{X}(T^2G) = \Image(\Hess)\oplus\Ker(\DIV)
$$
is an orthogonal decomposition where $\Ker(p_H)=\Ker(\DIV)=\Ker(a)$ and $\Image(p_H)=\Image(\Hess)=\Ker(a)^{\perp}.$
\end{theorem}

\begin{proof}
To see that $p_H$ is a projection note that,
$$
p_H^2 =(\Hess\circ\Delta_{Ch^2G}^{-1})\circ(\DIV\circ\Hess)\circ(\Delta_{Ch^2G}^{-1}\circ\DIV) = p_H.
$$
Since $\Hess$ is injective and $\Delta_{Ch^2G}$ is bijective we have $\Ker(p_H)=\Ker(\DIV)$ and $\Image(p_H)=\Image(\Hess).$ But by Proposition 5.4.2, $\Ker(\DIV)^{\perp}=\Image(\Hess)$, which shows that $p_H$ is orthogonal. By Proposition 5.3.1, $\Ker(\DIV)=\Ker(a)$ and this completes the proof.
\end{proof}

\begin{corollary}
(Second Canonical Form). Let $\gamma\colon DOp^2(G)\to\mathcal{X}(Ch^2(G))$ be the canonical isomorphism $\gamma(L)_i(j) =L(i,j)$ and let $c\colon DOp^2(G)\to\mathcal{X}(T^2(G))$ be the map $c=\Hess\circ\,\Delta_{Ch^2(G)}^{-1}\circ\gamma.$ Then $a\circ c$ is the identity on $DOp^2(G).$ In particular, $DOp^2(G)\cong\Image(\Hess).$
\end{corollary}

\begin{remark}
\textit{The Laplacian operator $\Delta_{Ch^2G}=\DIV\Hess$ is closely related to the  canonical self-adjoint, second order differential operator $\Div\circ\,\Delta_{tG}\circ\nabla$. To see it, note that,
$$
\Div\Delta_{tG}\nabla = \left(\Div\Div_{tG}\right)\circ\left(\nabla_{tG}\nabla\right)
$$
and we have $\Div\Div_{tG} Z(j)=\sum_{i\in V_G}\DIV Z_i(j)$ and $\nabla_{tG}\nabla\phi(i)=\Hess\zeta(\phi)$ where,
$$
\zeta(\phi)_i(j)=\phi(j)\thinspace -\negthickspace\sum_{k\in B(i,2)}\phi(k).
$$
Therefore,
\begin{align*}
\Div\Delta_{tG}\nabla\phi(j) & = \Div\Div_{tG}\left(\Hess\zeta(\phi)\right)(j)\\
& = \sum_{i\in V_G}\DIV\left(\Hess\zeta(\phi)\right)_i(j)\\
& = \sum_{i\in V_G}\Delta_{Ch^2(G)}\zeta(\phi)_i(j).
\end{align*}
There is an explicit formula for $\Div\Delta_{tG}\nabla\phi$ and one can derive one for $\Delta_{Ch^2(G)}\,\zeta$ based on the formula for $\DIV Z$ in Proposition 5.3.2, but it isn't particularly simple or revealing.}
\end{remark}

\mpar
There are no simple formulas for $\Delta_{Ch^2G}$ or its inverse so understanding the projection $p_H$ seems a bit remote. However, the operator $\Hess\circ\DIV$ acting on $\mathcal{X}(T^2G)$ is a relatively concrete object and it suggests a way to approximate $p_H$ via gradient descent. It's easy to see that $\ker(\DIV)$ consists of rest points for the flow generated by $\Hess\DIV,$ so it's reasonable to think that the initial point $W$ of the flow moves toward $\Ker(\DIV)= \Ker(p_H) = \Image(1- p_H)$ in which case one expects,
$$
p_H W = W-\lim_{t\to\infty} e^{-t\Hess\DIV} W.
$$
The following result confirms this expectation.

\begin{theorem}
Consider the initial value problem,  
$$
\tfrac{d}{dt}Z_t=\Hess\DIV Z_t,\, t > 0,\quad Z_0=W,
$$

\mpar
in the vector space $\mathcal{X}(T^2(G)).$ Let,
$$
W=p_HW+(1-p_H)W
$$

be the Helmholtz decomposition of the initial value. Then,
\begin{align*}
Z_t & =e^{-t\Hess\DIV}\, W\\
&  = \Hess e^{-t\Delta_{Ch^2(G)}}\left(\Delta_{Ch^2(G)}^{-1}\DIV W\right)+ (1-p_H) W
\end{align*}

is the unique solution of the initial value problem and,

$$
||(1-p_H) W- Z_t||_{L^2(\mathcal{X}(T^2(G)))}\leq e^{-t\lambda_0}||W||_{L^2(\mathcal{X}(T^2(G)))},
$$

\mpar
where $\lambda_0 >0$ is the smallest eigenvalue of $\Delta_{Ch^2(G)}$.
\end{theorem}

\begin{proof}
We're considering a linear ordinary differential equation on a finite dimensional inner product space so it can solved by exponentiation,
$$
Z_t =\sum_{n=0}^{\infty}\frac{(-tB)^n}{n!} W,
$$
where $B=\Hess\DIV$. If $\DIV W=0$ then, clearly, $Z_t\equiv W.$ On the other hand, let's show that for general initial values $W,$

$$
e^{-t\Hess\DIV}\, p_H W=\Hess e^{-t\Delta_{Ch^2(G)}}\left(\Delta_{Ch^2(G)}^{-1}\DIV W\right),
$$

\mpar
that is to say, the flow of $\Hess\DIV$ acting on $p_H W$ is the Hessian of the flow of $\Delta_{Ch^2(G)}$  acting on $\Delta_{Ch^2(G)}^{-1}\DIV W.$ To see it, note the equation is true for $t=0$ since it reduces to the definition of $p_H.$ Next, observe that if $Z_t$ is a solution in $\mathcal{X}(T^2(G))$ and $\zeta_t=\DIV Z_t$ then,

$$
\tfrac{d}{dt}\zeta_t = \DIV\tfrac{d}{dt}Z_t = -\DIV\Hess\DIV Z_t =-\Delta_{Ch^2(G)}\DIV Z_t = -\Delta_{Ch^2(G)}\,\zeta_t.
$$

\mpar

On the other hand, suppose $\zeta_t$ is a solution in $\mathcal{X}(Ch^2(G))$ and $Z_t=\Hess\zeta_t$ then,

$$
\tfrac{d}{dt}Z_t = \Hess\tfrac{d}{dt}\zeta_t = -\Hess\Delta_{Ch^2(G)}\zeta= -\Hess\DIV\Hess\zeta_t = -\Hess\DIV Z_t.
$$

\mpar
Let $W= p_H W + (1-p_H)W$ be the Helmholtz decomposition of the initial value. Then $\DIV(1-p_H)W=0$ and $p_H W=\Hess\eta$ for $\eta=\Delta_{Ch^2(G)}^{-1} \DIV W$ and it follows that,

$$
e^{-t\Hess\DIV} W=\Hess e^{-t\Delta_{Ch^2(G)}}\left(\Delta_{Ch^2(G)}^{-1}\DIV W\right)+ (1-p_H) W.
$$

Now, $\Delta_{Ch^2(G)}$ is a non-negative, self-adjoint, and invertible operator hence it has distinct eigenvalues with multiplicity,
$$
0<\lambda_0 < \lambda_1 < \cdots < \lambda_N,
$$
for some $N$, and corresponding eigenspaces which span $\mathcal{X}(Ch^2G).$ It is well-known that each eigenspace has an orthonormal basis and that eigenvectors of a self-adjoint operator in distinct eigenspaces are mutually orthogonal. Let $M$ be the dimension of $\mathcal{X}(Ch^2(G))$ and let $(\mu_m, \zeta_m), 1\leq m\leq M$ be an enumeration of an orthonormal basis of eigenvectors $\zeta$ and their associated eigenvalues $\mu$. Observe that in this notation, for every $1\leq m\leq M,\,\mu_m=\lambda_l$ for some $1\leq l\leq N.$

\mpar
Let $\eta=\Delta_{Ch^2(G)}^{-1}\DIV W$ and let $\eta=\sum_{m=1}^M c_m\zeta_m$ be its orthonormal expansion. Then,
\begin{align*}
||(1-p_H)W-Z_t||_2^2 & = ||\Hess e^{-t\Delta_{Ch^2(G)}}\left(\Delta_{Ch^2(G)}^{-1}\DIV W\right)||^2_2\\
& = \langle \Hess e^{-t\Delta_{Ch^2(G)}}\eta, \Hess e^{-t\Delta_{Ch^2(G)}}\eta\rangle\\
& = \sum_{m=1}^M c_m^2 e^{-2t\mu_m}\langle
\Hess\zeta_m, \Hess\zeta_m\rangle\\
& \leq e^{-2t\lambda_0}\sum_{m=1}^M \langle c_m\Hess\zeta_m, c_m\Hess \zeta_m\rangle\\
& = e^{-2t\lambda_0}||\Hess\eta||_2^2.
\end{align*}

But we have $\Hess\eta = p_H W $ and $||p_H W|| \leq ||W||$ which concludes the proof.
\end{proof}

\begin{remark}
\textit{Recall \cite{M2} the Helmholtz decomposition of vector fields on $G$ is implemented by the projection $p_{\nabla}\colon\mathcal{X}(G)\to \mathcal{X}(G),$ where $p_{\nabla}=\nabla\circ\Delta^{-1}\negthinspace\circ\Div.$ Using arguments similar to those above one can show,
$$
||(1-p_{\nabla})Y-X_t||_{L^2(\mathcal{X}(G))}\leq e^{-t\lambda_1}||Y||_{L^2(\mathcal{X}(G))},
$$ 
where 
$$
X_t=e^{-t\nabla\Div}Y=\nabla e^{-t\Delta}\left(\Delta^{-1}\Div Y\right) + (1-p_{\nabla})Y
$$
and $\lambda_1$ is the spectral gap of the Laplacian $\Delta.$}
\end{remark}

\section{Products of Vector Fields}
\mpar
We know that $DOp^2(G)$ is spanned by products of vector fields and therefore for every $X, Y\in\mathcal{X}(G)$ there exists $W\in \mathcal{X}(T^2(G))$ such that $XY=a(W).$ The next result gives a formula for one particular choice of $W.$ 

\mpar
Before stating the result it's helpful to recall the \textit{orientation} $\omega(\alpha)$ of a vertex $\alpha\in V_{t^2G}.$ If $\alpha = ij/jk, \,k\neq i$ is a forward translation then $\omega(\alpha)=+1,$ if $\alpha = ij/ki,\, k\neq j$ is a backward translation then $\omega(\alpha)=-1,$ and if $\alpha=ij/ji$ is a reflection then $\omega(\alpha)=0.$ Also recall the notation $\overline{u}=\sigma(u).$

\begin{proposition}
Let $S\colon\mathcal{X}(G)\to\mathcal{X}(G)$ be defined by the formula,
$$
SY(u)=Y(\overline{u})+Y(u) -\tfrac{1}{2}\left[\pi Y(\pi_+(u))+\pi Y(\pi(u))\right].
$$ For every $X,Y\in\mathcal{X}(G)$ define $X\negthinspace\circ\negthinspace Y\in\mathcal{X}(T^2G)$ by the rule,
$$ 
X\negthinspace\circ\negthinspace Y(\alpha) = 
\begin{cases}
X(\pi(\alpha))Y(\pi_+(\alpha)), & \omega(\alpha)=1,\\ X(\pi(\alpha))Y\left(\overline{\pi_+(\alpha)}\right), &\omega(\alpha)=-1\\ X(\pi(\alpha))SY(\pi(\alpha)), &\omega(\alpha)=0  
\end{cases}
$$

Then $a(X\negthinspace\circ\negthinspace Y)\phi(i)=X(Y\phi)(i)$ for all $\phi\in C(G).$
\end{proposition}

\begin{proof}
We have,
\begin{align*}
X(Y\phi)(i ) & = \sum_{\pi(u)=i} \negthickspace X(u)dY\phi(u)\\
& = \sum_{\pi(u)=i}\negthickspace X(u)\left(Y\phi(\pi_+(u))-Y\phi(i)\right)\\ 
& = \sum_{\pi(u)=i} \,\,\sum_{\pi(v)=\pi_+(u)} \negthickspace \negthickspace X(u)Y(v) d\phi(v) - \sum_{\pi(u)=i}\sum_{\pi(w)=i} \negthickspace X(u)Y(w) d\phi(w)\\ 
& = I+II.  
\end{align*}

In order to evaluate these terms it's helpful first to recall some facts about vertices in $t^2G.$ 
\begin{itemize}
\item if $u, v\in V_{tG}$ and $\pi_+(u)=\pi(v)$ then $\alpha= uv\in V_{t^2G}$
\item if $u,w\in V_{tG}$ and $\pi(u)=\pi(w)$ then $\alpha = u\overline{w}\in V_{t^2G}$ 
\item in the first case $d^2\phi(uv)=d\phi(v)-d\phi(u)$
\item in the second case  $d^2\phi(u\overline{w})= d\phi(\overline{w})-d\phi(u) = -(d\phi(w)+d\phi(u)).$ 
\end{itemize}

With this notation in hand we have,
\begin{align*} 
I & = \sum_{\pi(u)=i}\,\,\sum_{\pi(v)=\pi_+(u)} \negthickspace \negthickspace X(u)Y(v)[d\phi(v)\pm d\phi(u)]\\ 
& = \sum_{\pi(u)=i}\sum_{\pi(v)=\pi_+(u)}\negthickspace \negthickspace X(u)Y(v) d^2\phi(uv) + \sum_{\pi(u)=i} \negthickspace X(u)\pi Y(\pi_+(u)) d\phi(u),  
\end{align*}
and similarly,
\begin{align*}
II & = -\sum_{\pi(u)=i}\,\sum_{\pi(w)=i} \negthickspace X(u)Y(w)[d\phi(w)\pm d\phi(u)]\\ 
& = \sum_{\pi(u)=i}\,\sum_{\pi(v)=\pi_+(u)}\negthickspace \negthickspace X(u)Y(w) d^2\phi(u\overline{w}) + \sum_{\pi(u)=i} \negthickspace X(u)\pi Y(i)) d\phi(u).
\end{align*}

Recalling the fact that $d^2\phi(u\overline{u})=-2d\phi(u),$ we have,
\begin{align*}
I+II & \,= \negthinspace\sum_{\pi(u)=i}\,\,\sum_{\substack{\pi(v)=\pi_+(u)\\v\neq\overline{u}}}\negthickspace\negthickspace X(u)Y(v)d^2\phi(uv)  +  \sum_{\pi(u)=i}\,\,\sum_{\substack{\pi(w)=i\\w\neq u}}\negthickspace X(u)Y(\overline{w})d^2\phi(u\overline{w})\\ 
& \phantom{mmmm} +\sum_{\pi(u)=i} \negthickspace X(u)[Y(\overline{u}) +Y(u) -\tfrac{1}{2}\pi Y(\pi_+(u)) -\tfrac{1}{2}\pi Y(i)]d^2\phi(u\overline{u})\\ 
& = \sum_{\pi^2(\alpha)=i}\negthickspace X\negthinspace\circ\negthinspace Y(\alpha)d^2\phi(\alpha), 
\end{align*} and this completes the proof.
\end{proof}

This is a compact and appealing formula, although the addition of any element of $\Ker(a)$ yields an equivalent formula. Still, it has value as it allows us to draw an important conclusion about commutators of vector fields.
\mpar
\begin{corollary}
Let $X,Y\in\mathcal{X}(G).$ Then $[X, Y]\in\mathcal{X}(G)$ if and only if,
$$
\lozenge \left(X\negthinspace\circ\negthinspace Y\right)_i(j)=\lozenge \left(Y\negthinspace\circ\negthinspace X\right)_i(j),
$$
for every $i,j\in V_G$ such that $d(i,j)=2.$ Specifically, $[X,Y]$ is a vector field if and only if for all $d(i,j)=2,$
$$
\sum_{k\in V_G}A(i,k)A(k,j)X(ik)Y(kj) =\sum_{k\in V_G}A(i,k)A(k,j)Y(ik)X(kj)
$$
where $A$ is the adjacency matrix of $G.$
\end{corollary}

\begin{proof}
This is a straightforward consequence of the fact that the difference of the two terms above equals $\DIV\left(X\negthinspace\circ\negthinspace Y-Y\negthinspace\circ\negthinspace X\right)(j)$ when $d(i,j)=2.$ Thus, by Proposition 5.3.4, $[X,Y]=a(X\negthinspace\circ\negthinspace Y-Y\negthinspace\circ\negthinspace X)\in DOp^1(G)\cong\mathcal{X}(G).$
\end{proof}

\mpar
It's obvious that $X^*X$ is a self-adjoint, second order differential operator. This suggests calculating the adjoint of the product of two vector fields. Before doing so, it's helpful to recall that the adjoint of a vector field is given by the formula,
$$
\langle X\phi,\psi\rangle_{C(G)}=\langle \phi, \overline{X}\psi+ \psi\,\Div X\rangle_{C(G)}.
$$
Here've we've used the notation $\overline{X}=\sigma X$ to denote the vector field with coefficients $\overline{X}(ij)=X(\sigma(ij))=X(ji)$. Letting $m(\phi)$ stand for the multiplication operator $m(\phi)\psi(i)=\phi(i)\psi(i)$ we can write the adjoint of a vector field as,
$$
X^*=\overline{X}+m(\Div X).
$$
\begin{proposition}
Let $X,Y\in\mathcal{X}(G).$ Then,
$$
(XY)^* = \overline{Y}\,\overline{X} +\overline{Y}m(\Div X) +m(\Div Y)\overline{X} +m(\Div X\Div Y).
$$
Therefore, $(XY)^*\in DOp^2(G)$ if and only if $Y^*\Div X= 0.$
\end{proposition}

\begin{proof}
The proof of is a straightforward calculation based on the adjoint formula. It's clear that $(XY)^*$ is a second order operator so it is a differential operator if and only if $(XY)^*1= 0.$ But,
$$
(XY)^*1 = \overline{Y}\Div X +\Div X\Div Y = Y^*\Div X,
$$
which concludes the proof.
\end{proof}

\section{Operators Generalizing $\Div\circ\,\Delta_{tG}\circ\nabla$}

Observe that  $\Div\circ\,\Delta_{tG}\circ\nabla$ is an example of a wider class of operators where the vector field $\Delta_{tG}$ on $tG$ is replaced by a general vector field $Z.$  This suggests introducing the map,
$$
b\colon\mathcal{X}(tG)\to DOp^2(G)
$$ 
by the rule $b(Z)=\Div\circ\, Z\circ\nabla.$ Because $\mathcal{X}(tG)$ is canonically isomorphic to $\mathcal{X}(T^2G),$ we will think of $b$ as a map from second order vector fields to second order differential operators,
$$
b\colon\mathcal{X}(T^2G)\to DOp^2(G)
$$ 
Now, $a$ is a surjective map so for every $Z\in\mathcal{X}(T^2G)$ there exists $W\in\mathcal{X}(T^2G)$ such that $b(Z)=a(W).$ The following proposition presents formulas for $W$ and the adjoint of $b(Z)$ and derives some consequences from them.

\mpar
Before coming to the proposition let's review some notation. Recall that  $\sigma\colon tG\to tG$ is the map $\sigma(ij) = ji$ and $d\sigma\colon t^2G\to t^2 G$ is the map $d\sigma(ij/kl)= \sigma(ij)\sigma(kl).$ They act on the appropriate vector fields by the rules,
$$
\sigma X(u)=X(\sigma(u)) =X(\overline{u})=\overline{X}(u)
$$
and $d\sigma Z(\alpha)= Z(d\sigma(\alpha)).$ Finally, we use the notation $X\colon\negthinspace Y$ for the pointwise product $X\colon\negthinspace Y(u)= X(u) Y(u).$

\begin{proposition}
Let $b\colon\mathcal{X}(T^2G))\to DOp^2(G)$ be the operator, 
$$
b(Z)=\Div\circ Z\circ\nabla.
$$ 
Then,

\mpar
1. \quad $b(Z) = -a(Z+d\sigma Z),$

\mpar
2. \quad $|\Ker (b)| \geq |E_{tG}|$ and $|\Image(b)|\leq |E_{tG}|,$

\mpar
3. \quad $b(Z)^* = b(\sigma Z) - \left(\Div_{tG}Z +\sigma\Div_{tG}Z\right).$  

\mpar

\end{proposition}

\begin{proof}
For item 1 note that,
\begin{align*}
b(Z)\phi(i) & = \Div(Z\nabla\phi)\\
&  = \sum_{\pi(u)=i} [Z\nabla\phi(\overline{u}) - Z\nabla\phi(u)]\\
& = \sum_{\pi(u)=i}\sum_{\pi(\alpha)=\overline{u}}Z(\alpha) d^2\phi(\alpha) \, - \negthickspace\sum_{\pi(u)=i}\sum_{\pi(\alpha)=u}Z(\alpha) d^2\phi(\alpha).
\end{align*}
Observe the second term above is just $-a(Z)\phi(i).$ To evaluate the first term, note that $d^2\phi(d\sigma(\alpha)) = -d^2\phi(\alpha)$ and $\pi(\alpha)=\overline{u}$ if and only if $\pi(d\sigma(\alpha))=u.$ By a change of variables we have,
\begin{align*}
\sum_{\pi(u)=i}\sum_{\pi(\alpha)=\overline{u}}Z(\alpha)d^2\phi(\alpha) & = \sum_{\pi(u)=i}\sum_{\pi(\beta)=u}Z(d\sigma(\beta))d^2\phi(\sigma(d\beta))\\
& = -\sum_{\pi(u)=i}\sum_{\pi(\beta)=u}d\sigma Z(\beta)d^2\phi(\beta)\\
& = -a(d\sigma Z)\phi(i),
\end{align*}
as required. 

\mpar
Item 2 turns on the easily checked fact that $d\sigma$ is a fixed point free involution of $V_{t^2G}.$ Therefore, $\mathcal{X}(T^2G)$ splits into even and odd subspaces with respect to $d\sigma,$ each of which has dimension $|E_{tG}|.$ Evidently, the odd subspace is contained in $\Ker(b)$ which shows $|\Ker(b)|\geq |E_{tG}|$ and therefore $|\Image(b)|\leq |E_{tG}|.$

\mpar
For item 3, recall that the adjoint of $Z$ is $Z^*=\sigma Z + m(\Div_{tG} Z),$ considered as a vector field on $tG.$ Thus,
\begin{align*}
b(Z)^*\phi(i) & = \Div\circ\, Z^*\circ\nabla\phi(i)\\
& = \Div(\sigma Z\,\nabla\phi)(i) + \Div(\Div_{tG}Z\colon\negthickspace\nabla\phi)(i)\\
& = I+II.
\end{align*}
Evidently, $I= b(\sigma Z)\phi(i).$ To evaluate the second term note that $d\phi(\overline{u})= -d\phi(u).$  We have,
\begin{align*}
II & = \sum_{\pi(u)=i}[\Div_{tG}Z(\overline{u})d\phi(\overline{u}) - \Div_{tG} Z(u) d\phi(u)]\\
& = -\sum_{\pi(u)=i}[\Div_{tG}Z(\overline{u}) + \Div_{tG} Z(u)] d\phi(u)\\
& = -\sigma\Div_{tG}Z\phi(i) - \Div_{tG}Z\phi(i),
\end{align*}
which completes the proof.
\end{proof}

\begin{remark}
\textit{1. The adjoint formula says that $b(Z)^*$ and $b(\sigma Z)$ differ by twice the even part of the vector field $\Div_{tG}Z.$}

\mpar
\textit{2. It follows from the definitions that if $Z$ is self-adjoint or skew-adjoint as a vector field on $tG$ then $b(Z)$ is self adjoint or skew-adjoint as an operator in $DOp^2(G).$}

\mpar
\textit{3. We know that $\sigma Z = Z$ is a necessary and sufficient condition for $Z$ to be a self-adjoint vector field, hence $\sigma Z=Z$ implies $b(Z)^*=b(Z).$ (Note that the adjoint formula is consistent with the fact that if $Z$ is even then $\Div_{tG}=0$).}

\mpar
\textit{4. It's known that necessary and sufficient conditions that $Z$ is skew-adjoint are $\sigma Z=-Z$ and $\Div_{tG}Z=0.$ Note that $\Div_{tG}\sigma Z=-\Div_{tG} Z$ for any vector field, so the antisymmetry of $Z$ imposes no condition on $\Div_{tG}Z.$  Therefore, a slightly more general sufficient condition that $b(Z)$ is skew-adjoint is that $\sigma Z=-Z$ and $\Div_{tG}Z$ is odd, that is, $\sigma\Div Z=-\Div Z.$}
\end{remark}

\section{Adjoints}
Our final result is a formula for $a(Z)^*,$ the adjoint of the operator $a(Z).$ It's significantly more complicated than the formula for $b(Z)^*.$ The formula involves a complimentary notion of gradient of a vector field $X,$ called the acclivity, and it's adjoint, called the dispersion. Before going further, however, it's useful to look at $a(Z)^*$ expressed in terms of the transpose of $\DIV Z_i(j).$

\begin{lemma}
1. Let $\zeta\in\mathcal{X}(Lf^2G)$ and define its \textit{transpose} $\zeta^{\top}\in\mathcal{X}(Lf^2G)$ by the rule $\zeta^{\top}_i(j) = \zeta_j(i).$ Then,
$$
a(Z)^*\phi(i)=\sum_{j\in V_G}(\DIV Z)^{\top}_i(j)\phi(j).
$$

2. $a(Z)$ is self adjoint if and only if,
\begin{align*}
\lozenge Z_i(j) & =\lozenge Z_j(i), \quad d((i,j)=2\\
\lozenge Z_i(j)-\lozenge Z_i(j)  & = (\pi + d\pi)Z(ij) -(\pi + d\pi)Z(ji),\quad d(i,j)=1.
\end{align*}
3. We have,
$$
\sum_{j\in V_G}\DIV Z_i(j) = 0 \quad\text{and}\quad\sum_{i\in V_G}\DIV Z_i(j) =\Div\Div_{tG} Z(j).
$$

4. The adjoint can be expressed as $a(Z)^* = \Div\circ \, Z^*\circ \pi^*$ where $\pi^*\colon C(G)\to\mathcal{X}(G)$ is the operator $\pi^*\phi = \phi\circ\pi$ and $Z^*=\sigma Z+m(\Div_{tG} Z)$ is the adjoint of $Z,$ thought of as a vector field on $tG$ or, equivalently, as a first order operator on $\mathcal{X}(G).$

\mpar
\mpar

\end{lemma}

\begin{proof}
Item 1 is a familiar fact of matrix theory stated in an unfamiliar setting, so we include a proof here for completeness's sake. We have,
\begin{align*}
\langle a(Z)\phi, \psi\rangle_{C(G)} & = \sum_{i\in V_G}\psi(i)\sum_{j\in V_G}\DIV Z_i(j)\phi(j)\\
& = \sum_{j\in V_G}\phi(j)\sum_{i\in V_G}\DIV Z)_i(j)\psi(i)\\
& = \sum_{j\in V_G}\phi(j) \sum_{i\in V_G}(\DIV Z)^{\top}_j(i)\psi(i)  = \langle\phi, a(Z)^*\psi\rangle_{C(G)}.
\end{align*}

Item 2 follows by setting the difference $\DIV Z_i(j)-\DIV Z_j(i)$ equal to zero and equating terms.

\mpar
The sums in item 3 are $a(Z)1$ and $a(Z)^* 1,$ respectively. But $a(Z)1=0$ and,
\begin{align*}
\langle a(Z)\phi, 1\rangle_{C(G)} & = \sum_{i\in V_G}\sum_{\pi^2(\alpha)=i} Z(\alpha) d^2\phi(\alpha)\\ & = \langle Z, \nabla_{tG}\nabla\phi\rangle_{\mathcal{X}(T^2(G))} = \langle\Div\Div_{tG} Z, \phi\rangle_{C(G)}.
\end{align*}
For item 4, we can write,
\begin{align*}
a(Z)\phi(i) & = \sum_{\pi^2(\alpha)=i}Z(\alpha)d^2\phi(\alpha)\\
& = \sum_{\pi(u)=i}\sum_{\pi(\alpha)=u}Z(\alpha)[d\phi(\pi_+(\alpha) - d\phi(u)]\\
& = \sum_{\pi(u)=i} Z\nabla\phi(u) = \pi\circ Z\circ\nabla\phi(i),
\end{align*}
where $\pi\colon\mathcal{X}(G)\to C(G)$ is defined by the formula $\pi X(i) = \sum_{\pi(u)=i} X(u).$ Note that we are conflating $Z$, thought of as a second order vector field on $G$, with $Z$, thought of as a vector field on $tG.$ Recall this doesn't lead to error because $\mathcal{X}(T^2G)$ is canonically isomorphic to $\mathcal{X}(tG).$ But it's worthwhile noting it because taking the adjoint of $a(Z)$ in $DOp^2(G)$ involves taking the adjoint of $Z$ in $DOp^1(tG)$, that is, its adjoint as first order differential operator on $C(tG)\cong \mathcal{X}(G).$ Specifically,
$$
a(Z)^* = (\pi\circ Z\circ\nabla)^* = \Div\circ \,Z^*\circ\pi^*.
$$
Here $\langle ZX, Y\rangle_{\mathcal{X}(G)}=\langle X, Z^*Y\rangle_{\mathcal{X}(G)}$ and $\langle\pi X, \phi\rangle_{C(G)}=\langle x, \pi^*\phi\rangle_{\mathcal{X}(G)}$ are the formal adjoint operators. We've already seen the formula for the adjoint of a vector field on a graph and it's an easy exercise to show that $\pi^*\phi = \phi\circ \pi$ 
\end{proof}

\mpar
The definitions of acclivity and dispersion use the morphism $d\pi_+\colon t^2G\to tG,$ which is given by the rule,
$$
d\pi_+(ij/kl)=\pi_+(ij)\pi_+(kl)=jl.
$$ 
Later on we'll need the morphism $\sigma d\sigma\colon t^2G\to tG$ where, 
$$
\sigma d\sigma(ij/kl)=\sigma(\sigma(ij)\sigma(kl))=\sigma(ji/lk)=lk/ji
$$
and its action on vector fields, namely $\sigma d\sigma Z(\alpha)= Z(\sigma d\sigma(\alpha)).$ The cardinal feature of $\sigma d\sigma$ is that it reverses the order of appearance of vertices of $G$ in the vertices of $t^2G.$

\begin{definition}
1. The \textit{acclivity} is the operator $\bigtriangledown_{tG}\colon\mathcal{X}(G)\to\mathcal{X}(tG)$ defined by the rule,
$$
\bigtriangledown_{tG}X(\alpha) = \delta X(\alpha) = X(d\pi_+(\alpha)) - X(d\pi(\alpha)).
$$

2. The \textit{dispersion} is the operator $\Dsp_{tG}\colon\mathcal{X}(tG)\to\mathcal{X}(G)$ defined by the rule,
\begin{align*}
\Dsp_{tG}Z(u) & = \sum_{d\pi_+(\beta)=u}\negthickspace\negthickspace Z(\beta)\thickspace - \negthickspace \sum_{d\pi(\alpha)=u}\negthickspace\negthickspace Z(\alpha)\\
& = d\pi_+ Z(u)-d\pi Z(u).
\end{align*}
\end{definition}

\begin{remark}
\textit{It's helpful to compare the acclivity with the gradient. Let $X$ be a vector field on $G$ and let $\alpha = ij/kl$ be a vertex in $V_{t^2G}.$ Then,
\begin{align*}
\nabla_{tG} X(\alpha) & = dX(ij/kl) = X(kl)-X(ij)\\
\bigtriangledown_{tG}X(\alpha) & =\delta X(ij/kl) = X(jl) - X(ik).
\end{align*}}
\end{remark}

\mpar
The following result collects some elementary facts about $\nabla_{tG}, \bigtriangledown_{tG}$ and their adjoints.
\begin{proposition}
1. The acclivity and dispersion are mutually adjoint in the sense that,
$$
\langle \bigtriangledown_{tG}X, Z\rangle_{\mathcal{X}(tG)} = \langle X,\Dsp_{tG}Z\rangle_{\mathcal{X}(G)}.
$$

\mpar
2. The acclivity and and gradient agree on gradient vector fields, meaning,
$$
\nabla_{tG}\nabla\phi = \bigtriangledown_{tG}\nabla\phi
$$
for all functions $\phi\in C(G).$

\mpar
3. Define operators on $\mathcal{X}(G)$ by the formulas,
\begin{align*}
\Delta^{-,-} & = \Div_{tG}\nabla_{tG} \quad\quad\Delta^{-,+} = \Div_{tG}\bigtriangledown_{tG}\\
\Delta^{+,-} & = \Dsp_{tG}\nabla_{tG} \quad\quad\Delta^{+,+} = \Dsp_{tG}\bigtriangledown_{tG}.
\end{align*}
Then $\Delta^{-,-}$ and $\Delta^{+,+}$ are self-adjoint and $\Delta^{-,+}$ and $\Delta^{+,-}$ are mutually adjoint. 

\mpar
4.  For all $Z\in\mathcal{X}(T^2G),$ we have the identities,
\begin{align*}
\Div_{tG}\sigma Z & = -\Div_{tG}Z, \quad\quad \Div_{tG} d\sigma Z  = \sigma \Div_{tG} Z,\\
\Dsp_{tG}\sigma Z & = \sigma\Dsp_{tG} Z, \quad\quad\,\Dsp_{tG}d\sigma Z  = -\Dsp_{tG} Z. 
\end{align*}

5. Let $\theta\colon V_{t^2G}\to V_{t^2G}$ be defined by the formula,
$$
\theta(\alpha) =
\begin{cases}
\alpha, & \text{if $\omega({\alpha})=0,1$}\\
\sigma d\sigma(\alpha), & \text{if $\omega(\alpha) = -1.$}
\end{cases}
$$
Then $\theta$ extends to an involution of $\mathcal{X}(tG)$ by the rule $\theta Z(\alpha)= Z(\theta(\alpha))$ and,
$$
\Dsp_{tG} Z=\Div_{tG}\theta Z, \quad \Div_{tG}Z = \Dsp_{tG}\theta Z.
$$
\end{proposition}

\begin{proof}
By straightforward calculation,
\begin{align*}
\langle \bigtriangledown_{tG} X, Z\rangle_{\mathcal{X}(tG)} & = \sum_{\alpha\in V_{t^2G}}\negthickspace\delta X(\alpha) Z(\alpha)\\
& = \sum_{\alpha\in V_{t^2G}}\negthickspace [X(d\pi_+(\alpha)) - X(d\pi(\alpha))] Z(\alpha)\\
& =\langle d\pi_+^* X-d\pi^* X, Z\rangle_{\mathcal{X}(t^2G)}\\
& =\langle X, d\pi_+ Z-d\pi Z\rangle_{\mathcal{X}(tG)}\\
& = \langle X, \Dsp_{tG} Z\rangle_{\mathcal{X}(G)},
\end{align*}
which yields item 1. Observe that if $\alpha = ij/kl$ then,
\begin{align*}
\nabla_{tG}\nabla\phi(\alpha) & = [\phi(l)-\phi(k)]-[\phi(j)-\phi(i)]\\ 
& = [\phi(l)-\phi(j)]-[\phi(k)-\phi(i)] = \bigtriangledown_{tG}\nabla\phi(\alpha),
\end{align*}
which is item 2. For item 3 we have,
$$
\langle \Delta^{+,-}X , Y\rangle_{\mathcal{X}(G)} = \langle \nabla_{tG} X, \bigtriangledown_{tG} Y\rangle_{\mathcal{X}(tG)} = \langle X, \Delta^{-,+}y\rangle_{\mathcal{X}(G)},
$$
and the remaining cases are proved similarly. 

\mpar
All but the first identity in item 4 result from changes of variable in the sums defining divergence and dispersion. We prove the divergence identities first and then the dispersion identities. 

\mpar
Note that,
\begin{align*}
\Div_{tG}\sigma Z(u) & = \sum_{\pi(\alpha)=u}\negthickspace\sigma Z(\overline{\alpha}) -\sigma Z(\alpha)\\
& = \sum_{\pi(\alpha)=u}\negthickspace Z(\alpha)-Z(\overline{\alpha}) =-\Div_{tG} Z(u),
\end{align*}
which is the first divergence identity. 

\mpar
Next, observe that $\pi(\alpha)=u$ if and only if $\pi(d\sigma(\alpha))=\overline{u}.$ Note also that $\sigma$ and $d\sigma$  commute so that $d\sigma Z(\sigma(u)) = \sigma Z(d\sigma(u)).$ Thus,
\begin{align*}
\Div_{tG} d\sigma Z(u) & = \sum_{\pi(\alpha)=u} \negthickspace d\sigma Z(\overline{\alpha}) -d\sigma Z(\alpha)\\
& = \sum_{\pi(\alpha)=u}\negthickspace \sigma Z(d\sigma(\alpha))- Z(d\sigma(\alpha))\\
& = \sum_{\pi(\beta)=\overline{u}}\negthickspace\sigma Z(\beta)-Z(\beta)\\
& = \sum_{\pi(\beta)=\overline{u}} \negthickspace Z(\overline{\beta}) - Z(\beta)\\
& = \Div_{tG} Z(\overline{u})= \sigma\Div_{tG} Z(u).
\end{align*}

Proof of the dispersion identities relies on the facts that,
\begin{align}
d\pi_+(\alpha) = u  & \iff d\pi_+(\sigma(\alpha))=\overline{u}\iff d\pi(d\sigma(\alpha)) = \overline{u}\\
d\pi(\alpha) = u  & \iff \thickspace d\pi(\sigma(\alpha))=\overline{u}\thickspace\,\iff d\pi_+(d\sigma(\alpha))=u.
\end{align}
These facts are easy to verify. For example, $d\pi_+(\alpha)=ij$ if and only if $\alpha = *i/*j,$ if and only if $\sigma(\alpha)=*j/*i,$ if and only if $d\pi_+(\sigma(\alpha))=ji.$ The other cases are similar and so omitted.
We have,
\begin{align*}
\Dsp_{tG} \sigma Z(u) & = \sum_{d\pi_+(\alpha)=u}\negthickspace\negthickspace \sigma Z(\alpha)\,\, - \negthickspace\sum_{d\pi(\alpha)=u}\negthickspace\negthickspace\sigma Z(\alpha)\\
& = \sum_{d\pi_+(\beta)=\overline{u}}\negthickspace\negthickspace\sigma Z(\sigma(\beta))\,\, - \negthickspace\sum_{d\pi(\beta)=\overline{u}}\negthickspace\negthickspace\sigma Z(\sigma(\beta)) \\
& = \sum_{d\pi_+(\beta)=\overline{u}} \negthickspace\negthickspace Z(\beta)\,\, - \negthickspace\sum_{d\pi_(\beta)=\overline{u}}\negthickspace\negthickspace Z(\beta)\\
& = \Dsp_{tG} Z(\overline{u}) =\sigma\Dsp_{tG} Z(u).
\end{align*}

Finally, we have,
\begin{align*}
\Dsp_{tG} d\sigma Z(u) & = \sum_{d\pi_+(\alpha)=u}\negthickspace d\sigma Z(\alpha) \, - \negthickspace\sum_{d\pi(\alpha)=u} \negthickspace d\sigma Z(\alpha)\\
& = \sum_{d\pi(\beta)=u} \negthickspace d\sigma Z(d\sigma(\beta))\, - \negthickspace\sum_{d\pi_+(\beta)=u} \negthickspace d\sigma Z(d\sigma(\beta))\\
& = \sum_{d\pi(\beta)=u} \negthickspace Z(\beta) \,- \negthickspace\sum_{d\pi_+(\beta)=u}\negthickspace  Z(\beta) = -\Dsp_{tG} Z(u),
\end{align*}

It's elementary that $\theta^2$ is the identity, so $\theta$ extends to an involution of $\mathcal{X}(tG).$ The key to item 5, then, is the observation that $d\pi_+(\alpha)=\pi_+(\theta(\alpha))$ and $d\pi(\alpha)=\pi(\theta(\alpha)).$ If these identities are true, then the sums defining $\Dsp_{tG}Z$ equal the sums defining $\Div_{tG}\theta Z,$ and vice versa, by an elementary change of variables. 

\mpar
So let's check them. If $\omega(\alpha)=0,1$ then there are vertices $i,j, k\in V_G$ such $\alpha=ij/jk$ and we have $d\pi_+(\alpha) = jk= \pi_+(\alpha) =\pi_+(\theta(\alpha))$ and $d\pi(\alpha)=ij=\pi(\alpha)=\pi(\theta(\alpha)).$ On the other hand, if $\omega(\alpha)=-1$ and $\alpha = ij/ki$ then $\sigma d\sigma(\alpha) = ik/ji.$ Then we have $d\pi_+(\alpha) = ji = \pi_+(\sigma d\sigma(\alpha))=\pi_+(\theta(\alpha))$ and $d\pi(\alpha) = ik = \pi(\sigma d\sigma(\alpha)) =\pi(\theta(\alpha)),$ as required. 
\end{proof}

\begin{theorem}
The adjoint of $a(Z)\in DOp^2(G)$ is given by the formula,
$$
a(Z)^* = a(\sigma d\sigma Z) + \sigma\Dsp_{tG}Z +\sigma\Div_{tG} Z+ m(\Div\Div_{tG} Z).
$$

\begin{proof}
We offer two proofs of the formula based on the first and fourth parts of Lemma 8.1, beginning with $\DIV Z^\top.$  It's natural to guess that $a(\sigma d\sigma Z)$ is related to $a(Z)^*$ because $\sigma d\sigma(\alpha)$ is the reverse of $\alpha$ in the sense that $\sigma d\sigma(ij/kl) = lk/ji.$ This suggests calculating $\DIV(\sigma d\sigma Z)$ and comparing it to $(\DIV Z)^{\top}.$ 

\mpar
According to Proposition 5.3.2,
$$
\DIV(\sigma d\sigma Z)_i(j) = \lozenge(\sigma d\sigma Z)_i(j) + 
\begin{cases}
\pi^2(\sigma d\sigma Z)(i), & \text{$j=i,$}\\
-d\pi(\sigma d\sigma Z)(ij) - \pi(\sigma d\sigma Z)(ij), & \text{$d(i,j)=1,$}\\
0, & \text{$d(i,j)=2,$}
\end{cases}
$$
so we evaluate these terms one by one. First,
$$
\lozenge(\sigma d\sigma Z)_i(j) \thickspace = \negthickspace\sum_{\substack{\pi^2(\alpha)=i\\ \pi_+^2(\alpha)=j}}\negthickspace Z(\sigma d\sigma(\alpha)) \thickspace = \negthickspace\sum_{\substack{\pi^2(\beta)=j\\ \pi_+^2(\beta)=i}} \negthickspace Z(\beta) = \thickspace\lozenge Z_j(i),
$$
and,
$$
\pi^2(\sigma d\sigma Z)(i) = \sum_{\pi^2(\alpha)=i}Z(\sigma d\sigma(\alpha)) = \sum_{\pi_+^2(\beta)=i}Z(\beta) = \pi_+^2 Z(i).
$$
For the remaining terms, note that $d\pi(\alpha)=ik$ if and only if $d\pi_+(\sigma d\sigma(\alpha))=ki$ and $\pi(\alpha)=ij$ if and only if $\pi_+(\sigma d\sigma(\alpha)) = ji,$ as one easily checks by considering a generic vertex $\alpha = ij/kl.$ Then,
$$
d\pi(\sigma d\sigma Z)(ij)  = \sum_{d\pi(\alpha)=ij} Z(\sigma d\sigma(\alpha)) = \sum_{d\pi_+(\beta)=ji} Z(\beta) = d\pi_+ Z(ji),
$$
and,
$$
\pi(\sigma d\sigma Z)(ij)  = \sum_{\pi(\alpha)=ij} Z(\sigma d\sigma(\alpha)) = \sum_{\pi_+(\beta)=ji} Z(\beta) = \pi_+ Z(ji).
$$
Therefore, we have,
$$
\DIV(\sigma d\sigma Z)_i(j) = \lozenge Z_j(i) + 
\begin{cases}
\pi_+^2 Z(i), & \text{$j=i,$}\\
-d\pi_+Z(ji) - \pi_+Z(ji), & \text{$d(i,j)=1,$}\\
0, & \text{$d(i,j)=2.$}
\end{cases}
$$
On the other hand,
\begin{align*}
(\DIV Z)^\top_i(j) & = \DIV Z_j(i) \\
& = \lozenge(Z)_j(i) + 
\begin{cases}
\pi^2(Z)(i), & \text{$j=i,$}\\
-d\pi(Z)(ji) - \pi(\sigma d\sigma Z)(ji), & \text{$d(i,j)=1,$}\\
0, & \text{$d(i,j)=2,$}
\end{cases}
\end{align*}
and it follows that,
\begin{align*}
a(Z)^*\phi(i)- a(\sigma d\sigma Z)\phi(i) &  = \sum_{j\in V_G} (\DIV Z)_i^\top(j)\phi(i) - \sum_{j\in V_G} \DIV Z_i(j)\phi (j)\\
& = \sum_{j\in V_G} L(i,j)\phi(j)
\end{align*}
where,
$$
L(i,j) = \begin{cases}
\pi^2 Z(i)-\pi_+^2 Z(i), & j=i\\
[\pi_+ Z(ji)-\pi Z(ji)] + [d\pi_+ Z(ji)-d\pi Z(ji)]. & d(i,j)=1,\\
0, & d(i,j)=2.
\end{cases}
$$
Thus $L= a(Z)^*-a(\sigma d\sigma Z)$ is a first order operator, since the flux terms vanish.
\mpar
Recall that every first order operator has the form $L=X+m(\nu)$ for some $X\in\mathcal{X}(G)$ and $\nu\in C(G),$ specifically,
\begin{align*}
\nu(i) &=\sum_{j\in V_G} L(i,j)\\
X(u) & = L(\pi(u), \pi_+(u)),
\end{align*}

Thus,
\begin{align*}
\nu(i) & =  \pi^2 Z(i)-\pi_+^2 Z(i) +\sum_{\pi_+(u)=i}[\pi_+ Z(u)-\pi Z(u)] + [d\pi_+ Z(u)-d\pi Z(u)]\\
& = \pi^2 Z(i)-\pi_+^2 Z(i) +\pi_+^2 Z(i) -\pi_+\circ\pi Z(i) +\pi_+\circ d\pi_+ Z(i)-\pi_+\circ d\pi Z(i).
\end{align*}
Observe that if $\alpha=ij/kl$ then,
$$
\pi_+\circ d\pi(\alpha)=\pi_+(ik)=k=\pi\circ\pi_+(\alpha)
$$
and,
$$
\pi_+\circ d\pi_+(\alpha)=\pi_+(jl)=l=\pi_+^2(\alpha)
$$
and it follows that,
\begin{align*}
\nu(i)& =\pi_+^2 Z(i)-\pi\circ\pi_+ Z(i)-\pi_+\circ\pi Z(i)+\pi^2 Z(i)\\
&  =\Div\Div_{tG}Z(i).
\end{align*}
On the other hand,
\begin{align*}
X(u) & = [d\pi_+Z(\overline{u}) - d\pi Z(\overline{u})] + [\pi_+Z(\overline{u})-\pi Z(\overline{u})]\\
& = \Dsp_{tG} Z(\overline{u})+ \Div_{tG} Z(\overline{u})\\
& = \sigma\Dsp_{tG}Z(u)+\sigma\Div_{tG}Z(u),
\end{align*}

which establishes the formula.

\mpar
While our second proof is a bit more intricate, it has the virtue of calculating $a(Z)^*$ from first principles. Recall that the adjoint of $Z\in\mathcal{X}(tG)$ is $Z^*=\sigma Z+m(\Div_{tG}Z)$ and the adjoint of $\pi\colon\mathcal{X}(G)\to C(G)$ is $\pi^*\colon C(G)\to\mathcal{X}(G)$ where $\pi^*\phi(u) =\phi\circ\pi(u).$  Then,
\begin{align*}
a(Z)^*\phi(i) & = (\pi\circ Z\circ \nabla)^*\phi(i)\\
& = (\Div\circ \,Z^*\negthinspace\circ\pi^*)\phi(i)\\
& = \Div(\sigma Z(\pi^*\phi))(i) + \Div(\pi^*\phi\colon\negthickspace\Div_{tG} Z))(i)\\
& = I+II.
\end{align*}
Remember that $\pi^*\phi$ is a vector field on $G$ hence $Z(\pi^*\phi)$ is also a vector field on $G.$ Also remember that $\pi(\beta) = \overline{u}$ if and only if $\pi(d\sigma(\beta))=u,$ a fact we use in a change of variables argument below. We have,
\begin{align*}
I & = \sum_{\pi(u)=i}[\sigma Z(\pi^*\phi)(\overline{u}) - \sigma Z(\pi^*\phi)(u)]\\
& = \sum_{\pi(u)=i}\,\sum_{\pi(\beta)=\overline{u}}\sigma Z(\beta) d(\pi^*\phi)(\beta) - \sum_{\pi(u)=i}\,\sum_{\pi(\alpha)=u} \sigma Z(\alpha) d(\pi^*\phi)(\alpha)\\
& = \sum_{\pi(u)=i}\,\sum_{\pi(\alpha)=u}\sigma Z(d\sigma(\alpha)) d(\pi^*\phi)(d\sigma(\alpha)) - \sum_{\pi(u)=i}\,\sum_{\pi(\alpha)=u} \sigma Z(\alpha) d(\pi^*\phi)(\alpha)\\
& = \sum_{\pi(u)=i}\,[\negthinspace\sum_{\pi(\alpha)=u}\sigma d\sigma Z(\alpha) d(\pi_+^*\phi)(\alpha) - \sum_{\pi(\alpha)=u} \sigma Z(\alpha) d(\pi^*\phi)(\alpha)\,]\\
& = \sum_{\pi^2(\alpha)=i} [\sigma d\sigma Z(\alpha) d(\pi_+^*\phi)(\alpha) - \sigma Z(\alpha) d(\pi^*\phi)(\alpha)\,].
\end{align*}

In the second last line we used the fact that $d(\pi^*\phi)(d\sigma(\alpha)) = d\pi^*_+\phi(\alpha).$ To see it suppose $\alpha =uv.$ Then $d\sigma(\alpha) =\overline{u}\overline{v}$ and we have,
\begin{align*}
d\pi^*\phi(d\sigma(\alpha)) & = d\pi^*\phi(\overline{u}\overline{v}) = [\pi^*\phi(\overline{v}) - \pi^*\phi(\overline{u})]\\
& = \phi(\pi(\overline{v}))-\phi(\pi(\overline{u}))\\
& = \phi(\pi_+(v))-\phi(\pi_+(u)) = d\pi_+^*\phi(\alpha)).
\end{align*}

Adding and subtracting $\sigma d\sigma Z (d\pi^*\phi)(\alpha)$ gives,
\begin{align*}
I & = \sum_{\pi^2(\alpha)=i}(\sigma d\sigma Z(\alpha) [d\pi_+^*\phi(\alpha)-d\pi^*\phi(\alpha)] + [\sigma d\sigma Z(\alpha) - \sigma Z(\alpha)]d\pi^*\phi(\alpha))\\
& = \sum_{\pi^2(\alpha)=i}(\sigma d\sigma Z(\alpha)\bigtriangledown_{tG}\nabla\phi(\alpha) + \sum_{\pi(u)=i}\sum_{\pi(\alpha)=u}[\sigma d\sigma Z(\alpha) - \sigma Z(\alpha)]d\phi(d\pi(\alpha)]\\
& = III + IV.
\end{align*}

Now $\bigtriangledown_{tG}\nabla\phi(\alpha)=\nabla_{tG}\nabla\phi(\alpha) = d^2\phi(\alpha),$ hence $III=a(\sigma d\sigma Z)\phi(i).$ To calculate $IV$ observe that $\pi^2(\alpha)=\pi\circ d\pi(\alpha)$ hence,
\begin{align*}
\{\alpha\in V_{t^2G}\mid \pi^2(\alpha)=i\} & = \{\alpha\in V_{t^2G}\mid \pi(\alpha)=u, \pi(u)=i\}\\
&  = \{\alpha\in V_{t^2G}| d\pi(\alpha)= u, \pi(u)=i\}.
\end{align*}
We use this observation to justify a change of variables in the double sum below. Specifically,
\begin{align*}
IV & = \sum_{\pi(u)=i}\sum_{\pi(\alpha)=u}[\sigma d\sigma Z(\alpha) - \sigma Z(\alpha)]d\phi(d\pi(\alpha))\\
& = \sum_{\pi(u)=i}\sum_{ d\pi(\beta) = u}[\sigma d\sigma Z(\beta)-\sigma Z(\beta)]d\phi(u)\\
& = \sum_{\pi(u)=i}d\phi(u)\sum_{d\pi(\beta)=u}[\sigma Z(d\sigma(\beta))-\sigma Z(\beta)].
\end{align*}

Next, observe that $d\pi(\beta) = d\pi^+(d\sigma(\beta))$ or, equivalently, $d\pi_+(\alpha) = d\pi(d\sigma(\alpha)).$ To see it note that $d\pi(\beta)=ij$ if and only if $\beta = i*/j*,$ if and only if $d\sigma(\beta) = *i/*j,$ if and only if $d\pi^+(d\sigma(\beta)) = ij.$ It follows that,
\begin{align*}
\sum_{d\pi(\beta)=u}[\sigma Z(d\sigma(\beta))-\sigma Z(\beta)]  & = \sum_{d\pi^+(\alpha)=u}\sigma Z(\alpha) - \sum_{d\pi(\beta)=u}\sigma Z(\beta)\\
& =\Dsp_{tG}\sigma Z(u),
\end{align*}
and therefore,
$$
IV = \Dsp_{tG} \sigma Z\phi(i) = \sigma\Dsp_{tG}Z\phi(i),
$$
according to Proposition 8.4.4.

\mpar
Turning to term $II,$ we have,
\begin{align*}
II & = \Div (\pi^*\phi\colon\negthickspace\Div_{tG}Z)(i)\\
& = \sum_{\pi(u)=i}[\pi^*\phi(\overline{u})\Div_{tG} Z(\overline{u}) -\pi^*\phi(u)\Div_{tG} Z(u)]\\
& = \sum_{\pi(u)=i}[\phi(\pi(\overline{u}))\sigma\Div_{tG} Z(u) - \phi(\pi(u))\Div_{tG} Z(u) \pm \phi(i)\,\sigma \Div_{tG}Z(u)]\\
& =\sum_{\pi(u)=i}\sigma \Div Z(u)d\phi(u) + \phi(i)\sum_{\pi(u)=i}\Div_{tG} Z(\overline{u})-\Div_{tG} Z(u)\\
& = \sigma\Div_{tG} Z\phi(i) + \phi(i)\Div^2 Z_{tG}(i).
\end{align*}
The sum $I+II$ gives the desired formula.
\end{proof}
\end{theorem}

\begin{remark}
\textit{1. Note that the formula for the adjoint of the second order differential operator $a(Z)$ is a direct analogue of the formula for the adjoint of the first order operator differential  operator $X$, namely $X^*=\sigma X+m(\Div X).$ Evidently, $a(Z)^*\in DOp^2(G)$ if an only if $\Div\Div_{tG}=0,$ in perfect analogy with the first order case.}

\mpar
\textit{2. Observe that if $Z=\sigma d\sigma Z$ then the difference $a(Z)^*- a(Z)$ is a first order operator namely $\sigma\Dsp_{ tG} Z + \sigma\Div_{tG}Z + m(\Div\Div_{tG}Z).$ This operator does not vanish, in general, based solely on the symmetry condition $Z=\sigma d\sigma Z$. So, while the symmetry condition ensures the second order part of  $a(Z)^*- a(Z)$ vanishes it is not sufficient to prove $a(Z)^*=a(Z)$ or even that $a(z)^*\in DOp^2(G).$ }
\mpar
\textit{3.  It would be interesting to find a symmetry condition in addition to $\sigma d\sigma Z=Z$ and $\Div\Div_{tG}Z=0$ that guarantees $a(Z)$ is self-adjoint. A natural guess is $\theta Z=-Z$ because then $\Dsp_{tG}Z =-\Dsp_{tG}\theta Z =-\Div_{tG}Z,$ and so $\sigma\Dsp_{tG}Z+\sigma\Div_{tG}Z=0.$ However, a moment's thought reveals the conjunction $\sigma d\sigma Z=Z$ and $\theta Z=-Z$ implies $Z=0$, leading to a triviality.}
\end{remark}

\mpar
\appendix
\section{Tangent Graphs and Related Notions}
We recall definitions and results needed for the present work without proof but with explicit references to \cite{M1} and \cite{M2}.

\subsection{Tangent Graph}
Let $G=(V_G, E_G)$ be a finite simple graph. It's tangent graph $tG = (V_{tG}, E_{tG})$ has vertex set,
$$
V_{tG}=\{(i,j)\mid \{i,j\}\in E_G\}
$$ 
and edge set, 
$$
E_{tG}=\{\, \{(i,j), (k,l)\}\mid j=k\,\, \text{or}\,\, k=i\}.
$$

\mpar
Vertices of $tG$ are directed edges of $G$ where the ordered pair $(i,j)$ labels the edge $\{i,j\}$ directed from the base point $i$ to the end point $j.$ Two directed edges of $G$ form an edge in $tG$ provided the end point of one directed edge is the base point of the other directed edge.  (\textit{\cite{M1} Definition 2.1.5 and Remark 2.4.1;} \cite{M2} \textit{Definition 2.1 and Remark 2.3})

\mpar
We adopt the notation $ij=(i,j)$ for directed edges of $G.$ For every $u\in V_G$ is there is an edge $\{i,j\}\in E_G$ such that $u=ij.$ We often use concatenation of vertices, such as $u=ij,$ to denote vertices of $tG.$ Let $\pi, \pi_+\colon V_{tG}\to V_G$ be the projections $\pi(ij)=i$ and $\pi_+(ij)=j$ so that $\pi(u)$ is the base point of $u$ and $\pi_+(u)$ is the end point of $u$. Let $\sigma\colon V_{tG}\to V_{tG}$ be the involution $\sigma(ij)=ji,$ and note we frequently use the notation $\sigma(u)=\overline{u}.$ Then $\pi, \pi_+\colon tG\to G$ are graph homomorphisms and $\sigma\colon tG\to tG$ is also a graph homomorphism. (\cite{M1} \textit{Proposition 2.3.4;} \cite{M2} \textit{Proposition 2.2})

\mpar
Evidently, $tG$ is a kind of oriented line graph in as much as the vertices of the line graph $lG$ are the edges of $G$ while the vertices of $tG$ are the directed edges of $G$ and their respective adjacency relations involve incidence of edges. However, the two notions are not equivalent. In fact $tG$ determines $lG$ but not vice versa since there is a pair of  graphs whose line graphs are isomorphic but whose tangent graphs are not. That is, $tG$ is a strictly finer graph invariant than $lG.$ (\cite{M1} \textit{Proposition 2.12.3 and Remark 2.13.3})

\mpar\mpar
There is a pretty generalization of the handshaking lemma in $G,$ namely,
$$
\sum_{i\in V_G} \Deg(i)^2 = |E_{tG}|+|E_G|.
$$
(\cite{M1} \textit{Proposition 2.10})

\mpar
Since $tG$ is a finite simple graph it has a tangent graph denoted $t^2G=t(tG).$ If $\{u,v\}$ is an edge of $tG$ then $\alpha=uv$ is a vertex in $t^2G.$ Let $\pi, \pi_+$ and $\sigma$ denote the involution and projections relative to $t^2G$. Thus, $\pi(\alpha)=\pi(uv) =u, \, \pi_+(\alpha) = \pi_+(uv)=v,$ and $\sigma(\alpha)=\sigma(uv)=vu.$ In principle, we should use notation like $\pi_{tG},\, \pi_{tG,+},$ and $\sigma_{tG}$ to distinguish them from their conterparts on $G$ but we don't in order to avoid notational clutter. While this is potentially ambiguous, context usually makes usage clear. For example, since $u,v\in V_{tG}$ there exist edges $\{i,j\}, \{k.l\}\in E_G$ such that $u=ij$ and $v=kl.$ Since $\{u,v\}\in E_{tG}$ we have either $i=l$ or $j=k.$ Said differently, we have either $\pi(u)=\pi_+(v)$ or $\pi_+(u)=\pi(v).$ Thus $\alpha=uv=ij/kl$ and, 
$$
\pi^2(\alpha)=i,\,\, \pi_+\circ\pi(\alpha)=j,\,\,\pi\circ\pi_+(\alpha)=k,\,\, \pi_+^2(\alpha)=l.
$$

\mpar
There are three types of vertex in $t^2G$: forward translations $\alpha=ij/jk$ where $k\neq i;$ backward translations $\alpha = ij/ki$ where $k\neq i;$ and reflections $\alpha=ij/ji.$ We assign an orientation to each type and write $\omega(\alpha)=1, 0, -1$ accordingly as $\alpha$ is a forward translation, reflection, or backward translation. (\cite{M1} \textit{Figure 5})
\mpar
The tangent graph is a functor in the following sense. If $h\colon G\to H$ is a graph homomorphism (or morphism) then the differential of $h$ is the map $dh\colon V_{tG}\to V_{tH},$ where $dh(ij)= h(i)h(j),$ and $dh$ extends to a morphism from $tG$ to $tH.$ In particular, $d\pi, d\pi_+\colon t^2G\to tG$ and $d\sigma\colon t^2G\to t^2G$ are morphisms. Also, $d\sigma$ is an automorphism of $t^2G,$ namely $d\sigma(uv) = \sigma(u)\sigma(v) $ which commutes with $\sigma=\sigma{tG},$ the involution of the tangent graph. In fact, they are involutions which generate a group isomorphic to $\mathbb{Z}_2\times\mathbb{Z}_2.$(\cite{M1} \textit{Proposition 2.3.4 and Proposition 2.6})
 
\subsection{Vector Bundles}
For any graph $G$ let $e_i$ be the indicator function of $i\in V_G$ and so that $e_i(j) = 1$ if $j=i$ and $e_i(j)=0$ if $i\neq j.$ When the graph in question is the tangent graph $tG$ we use the notation $e_u$ for $u\in V_{tG}$ and when it is the second tangent graph $t^2G$ we write $e_\alpha$ for $\alpha\in V_{t^2G}.$

\mpar
The tangent space to $G$ at vertex $i$ is the vector space,
$$
T_i(G)=\langle e_u\mid u\in V_{tG},\,\pi(u)=i\rangle.
$$ 
If $X_i\in T_i(G)$ then $X_i=\sum_{\pi(u)=i}X(u)e_u$ where the constants $X(u)$ are called to coefficients of $X.$ The tangent space acquires an inner product by declaring $\langle e_u, e_v\rangle = e_u(v)$ so that,
$$
\langle X_i, Y_i\rangle_{T_i(G)} = \sum_{\pi(u)=i}\sum_{\pi(v)=i}X_i(u)Y_i(v)\langle e_u, e_v\rangle = \sum_{\pi(u)=i}X_i(u)Y_i(u).
$$

The tangent bundle of $G$ is the coproduct,
$$
T(G)=\coprod_{i\in V_G} T_i(G).
$$
It inherits an adjacency relation from $G$ by declaring $X\in T_i(G)$ and $Y\in T_j(G)$ are adjacent in $T(G)$ if and only if $i$ and $j$ are adjacent in $G$. (\cite{M1} \textit{Definition 3.1;} \cite{M2} \textit{Definition 3.4})

\mpar
A section of $T(G)$ is a map $X\colon V_G\to T(G)$ such that $X_i\in T_i(G)$ for all $i\in V_G.$ A vector field on $G$ is a section of the tangent bundle and has the form,
$$
X=\sum_{u\in V_{tG}} X(u)e_u,
$$ 
for some function $X\colon V_{tG}\to\mathbb{R}.$ Observe we have conflated $e_u\in T_{\pi(u)}(G)$ with its canonical injection into the coproduct $T(G).$ The vector space of all sections of $T(G)$ is denoted $\mathcal{X}(G).$ It inherits an inner product from its constituent tangent spaces by the rule,
$$
\langle X, Y\rangle_{\mathcal{X}(G)}=\sum_{i\in V_G}\langle X_i, Y_i\rangle_{T_i(G)}.
$$
Let $C(G)$ denote the space of real valued functions on $V_G.$ Evidently, a vector field defines a function on $V_{tG}$ and every element of $C(tG)$ defines a vector field by serving as its coefficients. Therefore $\mathcal{X}(G)\cong C(tG).$ (\cite{M1} \textit{Definition 3.1;} \cite{M2} \textit{Definition 3.4})

\mpar
A vector bundle $E$ on $G$ is the coproduct of a set of vector spaces indexed by $V_G.$ Specifically, if $E_i, i\in V_G$ is a finite set of vector spaces then,
$$
E(G)=\coprod_{i\in V_G}E_i
$$
Observe that $E$ inherits an adjacency relation from $G$ by declaring $F_i\in E_i$ and $F_j\in E_j$ are adjacent in $E$ if and only if $i$ and $j$ are adjacent in $G$. A section of $E$ is a function $F\colon V_G\to E$ such that $F_i\in E_i$ for all $i\in V_G.$ The vector space of all sections of $E$ is denoted $\mathcal{X}(E)$. If each $E_i$ has an inner product then $\mathcal{X}(E)$ inherits and inner product by the rule,
$$
\langle F, H\rangle_{\mathcal{X}(E)} = \sum_{i\in V_G}\langle F_i, H_i\rangle_{E_i}.
$$ 
The second tangent space of $G$ at $i$ is the vector space,
$$
T^2_i(G)=\langle e_\alpha\mid \alpha\in V_{t^2G},\, \pi^2(\alpha)=i\rangle.
$$
The corresponding bundle is the second tangent bundle, $T^2(G)$, sections of which are called second order vector fields. The space of second order vector fields is denoted $\mathcal{X}(T^2G).$ Each $T_i^2(G)$ has a natural inner product obtained by declaring that $\{e_\alpha \mid \pi^2(\alpha)=i\}$ is an orthonormal basis and $\mathcal{X}(T^2G)$ inherits an inner product from them. (\cite{M1} \textit{Definition 3.10})

\mpar
\textbf{Notation.} \textit{1. Strictly speaking, the space of vector fields should be denoted $\mathcal{X}(T(G))$ rather than $\mathcal{X}(G)$ but we adopt this notation to reduce notational clutter.}

\mpar
\textit{2. Functions on a graph can be multiplied pointwise and we write their product as $\phi\psi$. Functions can multiply vector fields and we write as their product as $\phi X$ or sometimes $\phi\cdot\negthinspace X,$ for emphasis' sake, where $(\phi X)(u)=\phi(\pi(u)) X(u).$ Note, however, there is some ambiguity when the functions are on $tG$ rather than $G.$ A function on $tG$ is canonically associated with a vector field on $G$ so the concatenation $X Y$ could mean their pointwise product as functions on $tG$ or their composition as operators on $C(G).$ To disambiguate we reserve concatenation to refer to composition: $XY\phi= X(Y\phi)$ and use the notation $X\colon\negthinspace Y$ to denote pointwise product: $X\colon\negthinspace Y(u)= X(u)Y(u).$}

\subsection{Operators}
Every linear transformation $L\colon C(G)\to C(G)$ has the form,
$$
L\phi(i) = \sum_{i\in V_G}L(i,j)\phi(j).
$$
We are interested in the structure of these operators as a function of distance in $G$ and we say that $L$ is of order $k$ provided $L(i,j)=0$ if $d(i,j)>k.$ Evidently, zeroth order operators are multiplier operators of the form,
$$
m(\nu)\phi(i) = \nu(i)\phi(i),
$$
for some function $\nu\in C(G).$ 

\mpar
We say $L$ is a differential operator provided $L 1(i) = \sum_{j\in V_G} L(i,j) = 0$ for all $i\in V_G.$ Every operator has a canonical decomposition as the sum of a differential operator and a multiplier operator, as follows. Let $\nu(i) = \sum_{j\in V_G} L(i,j)$ and let $L^\prime(i,j) = L(i,j)-\nu(i).$ Clearly, $L^\prime$ is a differential operator and $L=L^\prime + m(\nu).$

\mpar
The space of $k^{\textit{th}}$-order operators is denoted $Op^k(G)$ and the subspace of differential differential operators is denoted $DOp^k(G).$ First order differential operators are in one-to-one correspondence with vector fields by the rule,
$$
X(u)=L(\pi(u), \pi_+(u)),\, u\in V{tG},
$$
hence, $DOp^1(G)\cong\mathcal{X}(G).$ (\cite{M1} \textit{Definition 3.3;} \cite{M2} \textit{Proposition 2.7.4})

\mpar
Let $d\colon C(G)\to C(tG)$ be defined by the rule, $d\phi(u)=\phi(\pi_+(u))-\phi(\pi(u)).$ It's adjoint with respect to the natural inner products is $d^* = \pi_+ - \pi$ where, for every $f\in C(tG),$
$$
\pi_+f(i) = \sum_{\pi_+(u)=i}f(u)\quad\text{and}\quad \pi f(i)=\sum_{\pi(u)=i} f(u).
$$
This means, (\cite{M1} \textit{Proposition 3.2.1;} \cite{M2} \textit{Proposition 2.7.1}),
$$
\langle d\phi, f\rangle_{C(tG)} = \langle \phi, (\pi_+-\pi)f\rangle_{C(G)}.
$$
Note that since $\pi,\pi_+\colon C(tG)\to C(G)$ they have adjoints $\pi^*,\pi_+^*\colon C(G)\to C(tG)$ and and elementary calculation shows that are given by the rule $\pi^*\phi(u) =\phi\circ\pi(u)$ and $\pi_+^*\phi(u)=\phi\circ\pi_+(u).$ 

\mpar
Let the gradient $\nabla\colon  C(G)\to\mathcal{X}(G)$ and divergence $\Div\colon\mathcal{X}(G)\to C(G)$ be defined by the formulas,\ $\nabla\phi(u)=d\phi(u)$ and,
$$
\Div X_i = \pi_+X(i)-\pi X(i)=\negthickspace\sum_{\pi_+(u)=i} \negthickspace X(u) - \negthickspace\sum_{\pi(u)=i}\negthickspace X(u).
$$
Then (\cite{M1} \textit{Proposition 3.2.2;} \cite{M2} \textit{Proposition 2.7.2} ) the divergence and gradient are mutually adjoint in the sense that,
$$
\langle\nabla\phi, X\rangle_{\mathcal{X}(G)}=\langle\phi,\Div X\rangle_{C(G)}.
$$

The Laplacian is the operator $\Delta=\Div\circ\nabla\colon C(G)\to C(G).$ It is non-negative and self adjoint, $\Ker(\Delta)$ is the space of functions constant on each connected component of $G,$ and,
$$
\Delta\phi(i) = -2\sum_{\pi(u)=i} d\phi(u).
$$ 
Consequently, $\Delta$ is a first order differential operator associated with the constant vector field $X(u)=-2.$ (\cite{M1} \textit{Proposition 3.2.1,} \cite{M2} \textit{Proposition 2.7.3})

\mpar
Let $X$ be a vector field thought of as a first order differential operator. Then its adjoint in the sense that $\langle \phi, \psi\rangle_{C(G)} =\langle\phi, X^*\psi\rangle_{C(G)}$
is the operator,
$$
X^*=\sigma X + m(\Div X),
$$ 
where $\sigma X(u) = X(\sigma(u)) = X(\overline{u})$ (\cite{M2} \textit{Proposition 2.7.5}). We sometimes use the notation $\sigma X=\overline{X}$ and switch back and forth without special mention. Evidently, $X^*\in DOp^1(G)$ if and only if $X$ is a divergence free vector field.  Note that we will use the adjoint formula where the graph in question is $tG$ rather than $G$, namely $Z^*=\sigma Z+m(\Div_{tG} Z)$ where $Z\in\mathcal{X}(tG).$

\mpar
Since $tG$ is  finite simple graph, it has its own associated gradient $\nabla_{tG},$ divergence $\Div_{tG}, $ and Laplacian $\Delta_{tG}.$ It's a natural step to use the Laplacian on $\Delta_{tG}$ to define a second order differential operator on $G$, namely, $\Div\circ\Delta_{tG}\circ\nabla$ which is one of the motivating examples for the present work.

\mpar
\section{$t^2P_2$ and $tP_2^2$ are incommensurate}
If $\phi$ is a surjective homomorphism from $t^2P_2$ to $tP_2^2$ then the six sets $\phi^{-1}(u), u\in V_{tP_2^2}$ are non-empty, disjoint independent sets in $t^2P_2$ and they partition $V_{t^2P_2},$ a set with eight elements. So either four of these independent sets are singletons and two are doubletons or five are singletons and one is a tripleton. If these sets represent the vertices of $tP_2^2$ then there must be at least one edge in $t^2P_2$ between some pair of elements of these sets, if there is an edge in $tP_2^2$ between the vertices these sets represent. We argue by cases to show that none of the graphs formed by these partitions with these edge relations is a triangular wedge.  

\mpar
For the reader's convenience we reproduce part of Figure 2 to show the triangular wedge $tP_2^2$ and the cube $t^2P_2.$ Note that in the argument below we use the vertex labels of $a,b,c,d,e,f,g,h$ of the cube as defined in the caption.

\begin{figure}[h]
\begin{tikzpicture}
\draw[fill=black] (0.5,-4) circle (2pt);
\draw[fill=black] (3,-4) circle (2pt);
\draw[fill=black] (2,-3.5) circle (2pt);
\draw[fill=black] (0.5,-2.5) circle (2pt);
\draw[fill=black] (3,-2.5) circle (2pt);
\draw[fill=black] (2,-2) circle (2pt);

\node at (2, -1.7) {13};
\node at (2,-3.8) {31};
\node at (0.5,-2.2) {21};
\node at (0.5, -4.3) {12};
\node at (3,-2.2) {32};
\node at (3,-4.3) {23};

\draw[thick] (0.5,-4) -- (3,-4) -- (3, -2.5) -- (0.5,-2.5) -- (0.5, -4);
\draw[thick] (0.5, -4) -- (2,-3.5) -- (3, -4);
\draw[thick] (0.5, -2.5) -- (2, -2) -- (3, -2.5);
\draw[thick] (2,-3.5) -- (2, -2);

\draw[fill=black] (4.5,-4) circle (2pt);
\draw[fill=black] (6.5,-4) circle (2pt);
\draw[fill=black] (4.5,-2.5) circle (2pt);
\draw[fill=black] (6.5,-2.5) circle (2pt);
\draw[fill=black] (5.5,-3.5) circle (2pt);
\draw[fill=black] (7.5,-3.5) circle (2pt);
\draw[fill=black] (5.5,-2) circle (2pt);
\draw[fill=black] (7.5,-2) circle (2pt);

\node at (4.5,-4.3) {\textit{a}};
\node at (6.5,-4.3) {\textit{b}};
\node at (4.5,-2.2) {\textit{e}};
\node at (6.5,-2.25) {\textit{f}};
\node at (5.5,-3.8) {\textit{d}};
\node at (7.5, -3.8) {\textit{c}};
\node at (5.5, -1.7) {\textit{h}};
\node at (7.5, -1.7) {\textit{g}};

\draw[thick] (4.5,-4) -- (6.5,-4) -- (6.5, -2.5) -- (4.5, -2.5) -- (4.5,-4);
\draw[thick] (5.5,-3.5) -- (7.5,-3.5) -- (7.5, -2) -- (5.5, -2) -- (5.5,-3.5);
\draw[thick] (4.5, -4) -- (5.5,-3.5);
\draw[thick] (4.5, -2.5) -- (5.5,-2);
\draw[thick] (6.5, -4) -- (7.5,-3.5);
\draw[thick] (6.5, -2.5) -- (7.5 ,-2);
\end{tikzpicture}
\caption{$a=12/23,$ $b=23/32,$ $c=32/21,$ $d=21/12,$\\ $e=23/12,$ $f=32/23,$ $g=21/32,$ $h=12/21.$}
\end{figure}
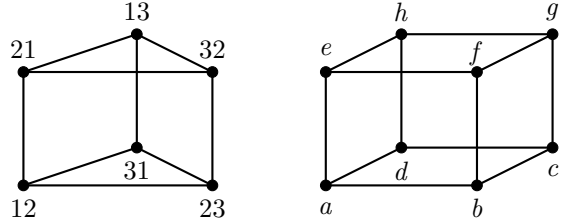

\mpar
We can assume $a$ is in a tripleton by symmetry of the cube. Then the remaining two vertices must come from the set $\{c, f, g, h\},$ since the tripleton is an independent set. But $g$ is a nearest neighbor of $c, f,$ and $h$ so it can't be in the tripleton, since then there would be an edge between two of its elements. There are three possibilities, $\{a, c, f\}, \{a,c, h\},$ and $\{a, f, h\}.$ By symmetry of the cube, the first two possibilities lead to the same graph so there are just two cases to check,
\begin{align*}
I \quad & \{a,c,f\}, b, d, e, g, h\\
II \quad & \{a, f, h\}, b, c, d, e, g
\end{align*}
The remaining cases involve choosing two pairs of diagonally opposite vertices. We use the symmetry of the cube to assume $a$ is in one of the doubletons and to identify cases leading to the same graphs. Observe $\{a,c\}, \{a, f\}, \{a,g\}$ and $\{a, h\}$ are the only possible doubletons containing $a.$ Note that $\{a, g\}$ is a diagonal of the cube whereas the other three are diagonals of a face of the cube, so by symmetry there are just two possibilities, say, $\{a, c\}$ and $\{a, g\}.$ There are subcases depending on the geometric relationship of the second doubleton to the one containing $a$.

\mpar
Suppose $\{a, c\}$ is one of the doubletons. Then the other doubleton can be in the same face, $\{b, d\},$ adjacent faces, $\{b,e\}, \{b,g\}, \{d,e\}, \{d, g\},$ opposite face, $\{e,g\},$ $\{f,h\},$ or diagonal of the cube, $\{b, h\}, \{d,f\}.$ By symmetry, $\{b,e\}$ and $\{d, g\}$ lead to the same graph as do $\{b,g\}$ and $\{d, e\}.$ Also by symmetry, the cube diagonals lead to the same graph. However, doubletons in the opposite face are not related by symmetry because one is parallel to $\{a,c\}$ and the other is anti-parallel. Thus, there are six more cases to check,
\begin{align*}
III\quad & \{a,c\}, \{b,d\}, e, f, g, h\\
IV\quad & \{a,c\}, \{b,e\}, d, f, g, h\\
V\quad &  \{a,c\}, \{b,g\}, d, e, f, h\\
VI\quad & \{a,c\}, \{e,g\}, b, d, f, h\\
VII\quad & \{a,c\}, \{f,h\}, b, d, e, g\\
VIII\quad & \{a, c\}, \{b, h\}, d, e, f, h
\end{align*}
Now suppose $\{a,g\}$ is one of the doubletons. The other doubletons can be diagonals of the cube, $\{b, h\}, \{c,e\}, \{d, f\},$ or diagonals of faces, $\{b,d\}, \{b,e\}, \{c,f\}, \{c,h\}, \{d,e\},$ $\{f,h\}.$ By symmetry, the cube diagonals lead to the same graph as do the three pairs of diagonals in opposite faces: $\{b,d\}, \{f,h\}$ and $\{b,e\}, \{c,h\}$ and $\{c,f\}, \{d,e\}$. Thus there are four additional cases to check,
\begin{align*}
IX\quad & \{a,g\}, \{b,h\}, c, d, e, f\\
X\quad & \{a,g\}, \{b, d\}, c, e, f, h\\
XI\quad & \{a,g\}, \{b,e\}, c, d, f, h\\
XII\quad & \{a,g\}, \{c,f\}, b, d, e, h
\end{align*}

and these exhaust all the possibilities.
\begin{figure}[h]
\begin{tikzpicture}
\draw[fill=black] (0,0.5) circle (2pt);
\draw[fill=black] (0.5,0) circle (2pt);
\draw[fill=black] (0.5,1) circle (2pt);
\draw[fill=black] (1.5,0) circle (2pt);
\draw[fill=black] (1.5,1) circle (2pt);
\draw[fill=black] (2,0.5) circle (2pt);
\node at (-0.3,0.5) {\textit{f}};
\node at (0.5,-0.3) {\textit{e}};
\node at (0.5, 1.3) {\textit{ag}};
\node at (1.5, -0.3) {\textit{d}};
\node at (1.5, 1.3) {\textit{bh}};
\node at (2.3,0.5) {\textit{c}};
\draw[thick] (1.5, 1) -- (0,0.5);
\draw[thick] (1.5, 1) -- (0.5,0);
\draw[thick] (1.5, 1) -- (1.5, 0);
\draw[thick] (1.5, 1) -- (2,0.5);
\draw[thick] (0.5, 0) -- (0,0.5);
\draw[thick] (1.5, 0) -- (2,0.5);
\draw[thick] (0.5, 1) -- (0,0.5);
\draw[thick] (0.5, 1) -- (0.5,0);
\draw[thick] (0.5, 1) -- (1.5, 0);
\draw[thick] (0.5, 1) -- (1.5, 1);
\draw[thick] (0.5, 1) -- (2, 0.5);

\draw[fill=black] (0,3) circle (2pt);
\draw[fill=black] (0.5,2.5) circle (2pt);
\draw[fill=black] (0.5,3.5) circle (2pt);
\draw[fill=black] (1.5,2.5) circle (2pt);
\draw[fill=black] (1.5,3.5) circle (2pt);
\draw[fill=black] (2,3) circle (2pt);
\node at (-0.3,3) {\textit{h}};
\node at (0.5,2.2) {\textit{f}};
\node at (0.5, 3.8) {\textit{ac}};
\node at (1.5,2.2) {\textit{e}};
\node at (1.5,3.8) {\textit{bg}};
\node at (2.3, 3) {\textit{d}};
\draw[thick] (0.5, 3.5) -- (1.5,3.5);
\draw[thick] (0.5, 3.5) -- (1.5,2.5);
\draw[thick] (0.5, 3.5) -- (2,3);
\draw[thick] (1.5, 3.5) -- (0, 3);
\draw[thick] (1.5, 3.5) -- (0.5, 2.5);
\draw[thick] (2,3) -- (0,3);
\draw[thick] (1.5, 2.5) -- (0.5, 2.5);
\draw[thick] (1.5, 2.5) -- (0, 3);

\draw[fill=black] (0,5.5) circle (2pt);
\draw[fill=black] (0.5,5) circle (2pt);
\draw[fill=black] (0.5,6) circle (2pt);
\draw[fill=black] (1.5,5) circle (2pt);
\draw[fill=black] (1.5,6) circle (2pt);
\draw[fill=black] (2,5.5) circle (2pt);
\node at (-0.3,5.5) {\textit{h}};
\node at (0.5, 4.7) {\textit{g}};
\node at (0.5, 6.3) {\textit{acf}};
\node at (1.5, 4.7) {\textit{e}};
\node at (1.5, 6.3) {\textit{b}};
\node at (2.3,5.5) {\textit{d}};
\draw[thick] (0.5, 6) -- (0.5, 5);
\draw[thick] (0.5, 6) -- (1.5, 5);
\draw[thick] (0.5, 6) -- (1.5, 6);
\draw[thick] (0.5, 6) -- (2, 5.5);
\draw[thick] (0,5.5) -- (0.5, 5);
\draw[thick] (0,5.5) -- (1.5, 5);
\draw[thick] (0,5.5) -- (2, 5.5);

\draw[fill=black] (3.1,0.5) circle (2pt);
\draw[fill=black] (3.6,0) circle (2pt);
\draw[fill=black] (3.6,1) circle (2pt);
\draw[fill=black] (4.6,0) circle (2pt);
\draw[fill=black] (4.6,1) circle (2pt);
\draw[fill=black] (5.1,0.5) circle (2pt);
\node at (2.8,0.5) {\textit{h}};
\node at (3.6,-0.3) {\textit{f}};
\node at (3.6, 1.3) {\textit{ag}};
\node at (4.6, -0.3) {\textit{e}};
\node at (4.6, 1.3) {\textit{bd}};
\node at (5.4,0.5) {\textit{c}};
\draw[thick] (4.6, 1) -- (3.1, 0.5);
\draw[thick] (4.6, 1) -- (3.6, 0);
\draw[thick] (4.6, 1) -- (5.1, 0.5);
\draw[thick] (3.6, 0) -- (3.6, 1);
\draw[thick] (3.6, 0) -- (4.6, 1);
\draw[thick] (3.6, 0) -- (4.6, 0);
\draw[thick] (4.6, 0) -- (3.1, 0.5);
\draw[thick] (3.6, 1) -- (3.1, 0.5);
\draw[thick] (3.6, 1) -- (3.6, 0);
\draw[thick] (3.6, 1) -- (4.6, 0);
\draw[thick] (3.6, 1) -- (4.6, 1);
\draw[thick] (3.6, 1) -- (5.1,0.5);

\draw[fill=black] (3.1,3) circle (2pt);
\draw[fill=black] (3.6,2.5) circle (2pt);
\draw[fill=black] (3.6,3.5) circle (2pt);
\draw[fill=black] (4.6, 2.5) circle (2pt);
\draw[fill=black] (4.6, 3.5) circle (2pt);
\draw[fill=black] (5.1,3) circle (2pt);
\node at (2.8,3) {\textit{h}};
\node at (3.6,2.2) {\textit{f}};
\node at (3.6, 3.8) {\textit{ac}};
\node at (4.6, 2.2) {\textit{d}};
\node at (4.6, 3.8) {\textit{eg}};
\node at (5.4,3) {\textit{b}};
\draw[thick] (3.6, 3.5) -- (4.6,3.5);
\draw[thick] (3.6, 3.5) -- (4.6,2.5);
\draw[thick] (3.6, 3.5) -- (5.1,3);
\draw[thick] (4.6, 3.5) -- (3.1,3);
\draw[thick] (4.6, 3.5) -- (3.6,2.5);
\draw[thick] (3.1, 3) -- (4.6,2.5);
\draw[thick] (3.6, 2.5) -- (5.1,3);

\draw[fill=black] (3.1,5.5) circle (2pt);
\draw[fill=black] (3.6,5) circle (2pt);
\draw[fill=black] (3.6,6) circle (2pt);
\draw[fill=black] (4.6,5) circle (2pt);
\draw[fill=black] (4.6,6) circle (2pt);
\draw[fill=black] (5.1,5.5) circle (2pt);
\node at (2.8,5.5) {\textit{g}};
\node at (3.6, 4.7) {\textit{e}};
\node at (3.6, 6.3) {\textit{afh}};
\node at (4.6, 4.7) {\textit{d}};
\node at (4.6, 6.3) {\textit{b}};
\node at (5.4,5.5) {\textit{c}};
\draw[thick] (3.6, 6) -- (3.1, 5.5);
\draw[thick] (3.6, 6) -- (3.6, 5);
\draw[thick] (3.6, 6) -- (4.6, 5);
\draw[thick] (3.6, 6) -- (4.6, 6);
\draw[thick] (3.6, 6) -- (5.1, 5.5);
\draw[thick] (4.6, 6) -- (5.1, 5.5);
\draw[thick] (3.1, 5.5) -- (5.1, 5.5);
\draw[thick] (5.1, 5.5) -- (4.6, 5);

\draw[fill=black] (6.2,0.5) circle (2pt);
\draw[fill=black] (6.7,0) circle (2pt);
\draw[fill=black] (6.7,1) circle (2pt);
\draw[fill=black] (7.7,0) circle (2pt);
\draw[fill=black] (7.7,1) circle (2pt);
\draw[fill=black] (8.2,0.5) circle (2pt);
\node at (5.9,0.5) {\textit{h}};
\node at (6.7,-0.3) {\textit{f}};
\node at (6.7, 1.3) {\textit{ag}};
\node at (7.7, -0.3) {\textit{d}};
\node at (7.7, 1.3) {\textit{be}};
\node at (8.5,0.5) {\textit{c}};
\draw[thick] (6.7, 1) -- (6.2, 0.5);
\draw[thick] (6.7, 1) -- (6.7, 0);
\draw[thick] (6.7, 1) -- (7.7, 0);
\draw[thick] (6.7, 1) -- (7.7, 1);
\draw[thick] (6.7, 1) -- (8.2, 0.5);
\draw[thick] (7.7, 1) -- (6.2, 0.5);
\draw[thick] (7.7, 1) -- (6.7, 0);
\draw[thick] (7.7, 1) -- (8.2, 0.5);
\draw[thick] (7.7, 0) -- (6.2, 0.5);
\draw[thick] (7.7, 0) -- (8.2, 0.5);

\draw[fill=black] (6.2,3) circle (2pt);
\draw[fill=black] (6.7,2.5) circle (2pt);
\draw[fill=black] (6.7,3.5) circle (2pt);
\draw[fill=black] (7.7,2.5) circle (2pt);
\draw[fill=black] (7.7,3.5) circle (2pt);
\draw[fill=black] (8.2,3) circle (2pt);
\node at (5.9,3) {\textit{g}};
\node at (6.7,2.2) {\textit{e}};
\node at (6.7, 3.8) {\textit{ac}};
\node at (7.7, 2.2) {\textit{d}};
\node at (7.7, 3.8) {\textit{fh}};
\node at (8.5,3) {\textit{b}};
\draw[thick] (6.7, 3.5) -- (6.2, 3);
\draw[thick] (6.7, 3.5) -- (6.7, 2.5);
\draw[thick] (6.7, 3.5) -- (7.7, 2.5);
\draw[thick] (6.7, 3.5) -- (8.2, 3);
\draw[thick] (7.7, 3.5) -- (6.2, 3);
\draw[thick] (7.7, 3.5) -- (6.7, 2.5);
\draw[thick] (7.7, 3.5) -- (7.7, 2.5);
\draw[thick] (7.7, 3.5) -- (8.2, 3);

\draw[fill=black] (6.2,5.5) circle (2pt);
\draw[fill=black] (6.7,5) circle (2pt);
\draw[fill=black] (6.7,6) circle (2pt);
\draw[fill=black] (7.7,5) circle (2pt);
\draw[fill=black] (7.7,6) circle (2pt);
\draw[fill=black] (8.2,5.5) circle (2pt);
\node at (5.9,5.5) {\textit{h}};
\node at (6.7, 4.7) {\textit{g}};
\node at (6.7, 6.3) {\textit{ac}};
\node at (7.7, 4.7) {\textit{f}};
\node at (7.7, 6.3) {\textit{bd}};
\node at (8.5,5.5) {\textit{e}};
\draw[thick] (6.7, 6) -- (8.2, 5.5);
\draw[thick] (6.7, 6) -- (6.7, 5);
\draw[thick] (6.7, 6) -- (7.7, 6);
\draw[thick] (7.7, 6) -- (6.2, 5.5);
\draw[thick] (7.7, 6) -- (7.7, 5);
\draw[thick] (8.2, 5.5) -- (6.2, 5.5);
\draw[thick] (8.2, 5.5) -- (7.7, 5);
\draw[thick] (6.7, 5) -- (6.2, 5.5);
\draw[thick] (6.7, 5) -- (7.7, 5);

\draw[fill=black] (9.3,0.5) circle (2pt);
\draw[fill=black] (9.8,0) circle (2pt);
\draw[fill=black] (9.8,1) circle (2pt);
\draw[fill=black] (10.8,0) circle (2pt);
\draw[fill=black] (10.8,1) circle (2pt);
\draw[fill=black] (11.3,0.5) circle (2pt);
\node at (9,0.5) {\textit{h}};
\node at (9.8,-0.3) {\textit{e}};
\node at (9.8, 1.3) {\textit{ag}};
\node at (10.8, -0.3) {\textit{d}};
\node at (10.8, 1.3) {\textit{cf}};
\node at (11.6,0.5) {\textit{b}};
\draw[thick] (9.3,0.5) -- (9.8, 0);
\draw[thick] (9.3,0.5) -- (10.8, 0);
\draw[thick] (9.3,0.5) -- (9.8, 1);
\draw[thick] (9.8,0) -- (9.8, 1);
\draw[thick] (9.8,0) -- (10.8, 0);
\draw[thick] (9.8,1) -- (10.8, 0);
\draw[thick] (9.8,1) -- (10.8, 1);
\draw[thick] (9.8,1) -- (11.3, 0.5);
\draw[thick] (10.8,0) -- (10.8, 1);
\draw[thick] (10.8,0) -- (11.3, 0.5);
\draw[thick] (10.8,1) -- (11.3, 0.5);

\draw[fill=black] (9.3,3) circle (2pt);
\draw[fill=black] (9.8,2.5) circle (2pt);
\draw[fill=black] (9.8,3.5) circle (2pt);
\draw[fill=black] (10.8,2.5) circle (2pt);
\draw[fill=black] (10.8,3.5) circle (2pt);
\draw[fill=black] (11.3,3) circle (2pt);
\node at (9,3) {\textit{g}};
\node at (9.8,2.2) {\textit{f}};
\node at (9.8, 3.8) {\textit{ac}};
\node at (10.8, 2.2) {\textit{e}};
\node at (10.8, 3.8) {\textit{bh}};
\node at (11.6,3) {\textit{d}};
\draw[thick] (9.8, 3.5) -- (9.3, 3);
\draw[thick] (9.8, 3.5) -- (10.8, 2.5);
\draw[thick] (9.8, 3.5) -- (10.8, 3.5);
\draw[thick] (9.8, 3.5) -- (11.3, 3);
\draw[thick] (10.8, 3.5) -- (9.3, 3);
\draw[thick] (10.8, 3.5) -- (9.8, 2.5);
\draw[thick] (10.8, 3.5) -- (10.8, 2.5);
\draw[thick] (10.8, 3.5) -- (11.3, 3);
\draw[thick] (9.8, 2.5) -- (10.8, 2.5);
\draw[thick] (9.8, 2.5) -- (9.3, 3);

\draw[fill=black] (9.3,5.5) circle (2pt);
\draw[fill=black] (9.8,5) circle (2pt);
\draw[fill=black] (9.8,6) circle (2pt);
\draw[fill=black] (10.8,5) circle (2pt);
\draw[fill=black] (10.8,6) circle (2pt);
\draw[fill=black] (11.3,5.5) circle (2pt);
\node at (9,5.5) {\textit{h}};
\node at (9.8, 4.7) {\textit{g}};
\node at (9.8, 6.3) {\textit{ac}};
\node at (10.8, 4.7) {\textit{f}};
\node at (10.8, 6.3) {\textit{be}};
\node at (11.5,5.5) {\textit{d}};
\draw[thick] (9.8, 6) -- (9.8, 5);
\draw[thick] (9.8, 6) -- (10.8, 6);
\draw[thick] (9.8, 6) -- (11.3, 5.5);
\draw[thick] (10.8, 6) -- (9.3, 5.5);
\draw[thick] (10.8, 6) -- (10.8, 5);
\draw[thick] (9.3, 5.5) -- (9.8, 5);
\draw[thick] (9.3, 5.5) -- (11.3, 5.5);
\draw[thick] (9.8, 5) -- (10.8, 5);

\end{tikzpicture}
\caption{Graphs $I$ to $XII$ are arranged in typewriter fashion from top left to bottom right. Labels like $\{a, c, f\}$ are shortened to $acf$ to reduce visual clutter.}
\end{figure}
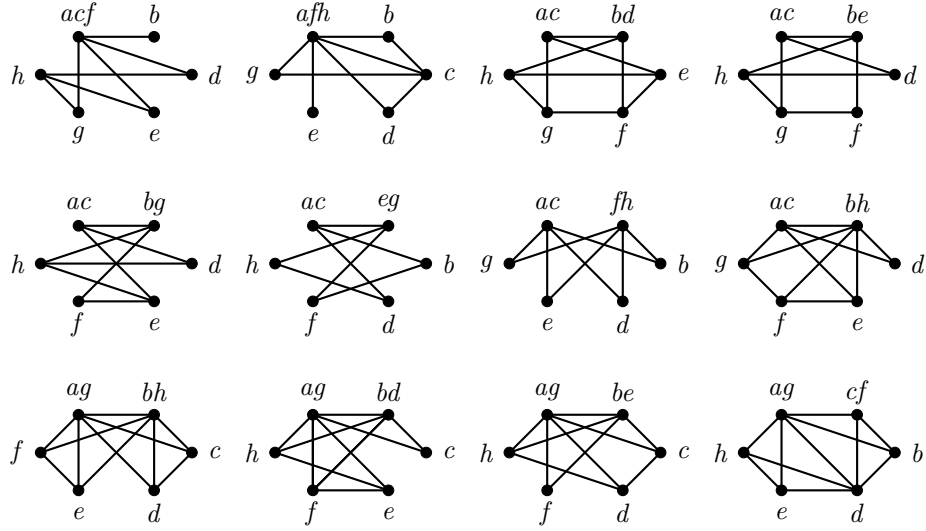

\mpar
In Figure 6 these twelve graphs are displayed in a standard format by arranging the vertices in a regular hexagon, with vertices labelled counterclockise by the elements of the partition, as listed in the order above.\mpar
In every case but graph $III$ there is a vertex of degree other than $3$ hence they are not isomorphic to $tP_2^2.$ One sees by inspection that $III$ is the complete bipartite graph $K_{3,3}$ formed by the parts $\{ac, f, h\}$ and $\{bd, g, h\}.$ It has 9 edges and 6 vertices of degree 3, just as $tP_2^2$ does. Distinguishing them is not elementary as it depends on the fact that genus is an isomorphism invariant. To see they are not isomorphic observe that $tP_2^2,$ the triangular wedge, can be embedded on a sphere so it has genus 0. On the other hand, $K_{3,3}$ can be embedded on a torus but not on a sphere, so it has genus 1.

\begin{remark}
\textit{1. It's obvious that there is no surjection from $tP_2^2$ to $t^2P_2$ and just as obvious that neither $t^2P_2$ nor $tP_2^2$ can be a subgraph of the other. So, the example shows that $tP_2^2$ and $t^2P_2$ are incommensurate.}

\mpar
\textit{2. It's straightforward to check that $K_{3,3}=tK_{1,3}.$ So, there is a surjective homomorphism from $t^2P_2$, the second tangent graph of the path of length two, to $tK_{1,3}$, the tangent graph of the star graph $K_{1,3}.$ It's not clear whether this is a coincidence or part of a pattern.}
\end{remark}

\end{document}